\documentclass[review,hidelinks,onefignum,onetabnum]{siamart220329}

\usepackage{lipsum}
\usepackage{amsfonts}
\usepackage{graphicx}
\usepackage{epstopdf}
\usepackage{algorithmic}

\ifpdf
  \DeclareGraphicsExtensions{.eps,.pdf,.png,.jpg}
\else
  \DeclareGraphicsExtensions{.eps}
\fi

\DeclareMathOperator*{\argmin}{arg\,min}

\newsiamremark{remark}{Remark}
\newsiamremark{hypothesis}{Hypothesis}
\crefname{hypothesis}{Hypothesis}{Hypotheses}
\newsiamthm{claim}{Claim}

\headers{Statistical and Algebraic Properties of the 4-Laplacian }{ I. Al-Awamleh, and R. Smits}

\title{Statistical and Algebraic Properties for the 4-Laplacian via Averaging\thanks{Submitted to the editors DATE.}}

\author{Ishraq AlAwamleh\thanks{Department of Mathematical Sciences, New Mexico State University, Las Cruces, New Mexico
1290 Frenger Mall, Science Hall 236
, 88003-8001 (\email{ishraq@nmsu.edu}, \url{https://math.nmsu.edu}).}
\and Robert Smits\thanks{Department of Mathematical Sciences, New Mexico State University, Las Cruces, New Mexico
1290 Frenger Mall, Science Hall 236
, 88003-8001 (\email{rsmits@nmsu.edu}, \url{https://math.nmsu.edu}).}}

\usepackage{amsopn}

\newtheorem{example}{Example}

\ifpdf
\hypersetup{
  pdftitle={Statistical and Algebraic Properties for the 4-Laplacian},
  pdfauthor={ I. Al-Awamleh, and R. Smits}
}
\fi

\begin{document}

\maketitle

\begin{abstract}
We find discrete analogs to continuous mean value principles that are used in the numerical analysis of the normalized p-Laplacian for particular values of p.  When $p=2$ these have analogs both to difference schemes for the Laplace on the cubic lattice and in the complex plane, showing discrete mean values hold for polynomials over any polygons with enough vertices.  A particularly interesting case is the $4-Laplacian$ which is the harmonic mean between the usual Laplacian and $\infty$-Laplacian. The averaging process for $p=4$ is an extension of typical measures of central tendency, the mean, median and midrange and has an exact formula in terms of the mean, standard deviations and skewness.  Finally, we show the existence of mean value principles for all exceptional regular polytopes.  
	
\end{abstract}



\section{Introduction}

\label{sec:introd}

It is a classical theorem that a function $u$ is harmonic on a domain $\Omega$ if and only if for all arbitrarily small balls $ B_{r}\left(x_{0}\right)$ contained in a region $\Omega$ it takes on its mean values, namely

\begin{equation}\label{eq1.1}
  u(x_0)=\frac{\int_{ B_{r}\left(x_{0}\right)} u (x) d x}{\int_{ B_{r}\left(x_{0}\right)}1 d x  } =\frac{1}{\left|B_{r}(x_0)\right|} \int_{B_{r}(x_0)} u(x) d x
 \end{equation}

It is also well known that this characterization can be loosened to only requiring 

\begin{equation}\label{eq1.2}
 u(x)=\frac{1}{\left|B_{\varepsilon}(x)\right|} \int_{B_{\varepsilon}(x)} u(y) d y+o\left(\varepsilon^2\right),
 \end{equation}
as $\varepsilon \rightarrow 0 $, see the paper of Manfredi, Parviainen and Rossi \cite{ManParRoss}.  Beginning with that paper an enormous amount of results connecting asymptotic mean value properties to solutions of nonlinear elliptic and parabolic problems, especially the p-Laplacian have occurred. In two dimension a number of mean value property involving gradients  were shown in \cite{GiorgiSmits}, motivated by numerical analysis of a variant of the p-Laplacian, namely 

\begin{equation}\label{eq1.3}
 u\left(x_0\right)=\frac{\int_{B_\epsilon\left(x_0\right)}\left|\langle \nabla u\left(x_0\right) ,\left(x-x_0\right)\rangle\right|^{p-2} u(x) d x}{\int_{B_\epsilon\left(x_0\right)}\left|\langle \nabla u\left(x_0\right) ,\left(x-x_0\right)\rangle \right|^{p-2} d x}+o\left(\epsilon^2\right)
 \end{equation}
holds in a viscosity sense, if and only if, $\Delta_p u\left(x_0\right)=0$ holds in a viscosity sense where $\Delta_p u=\operatorname{div}\left(|\nabla u|^{p-2} \nabla u\right)$ is the $p$-Laplace operator.

\newpage
These results were extended to all dimensions in a recent paper of one of the coauthors \cite{Alawamleh}.  In general for smooth functions with nonvanishing gradient, there are expansion theorems of the form

\begin{equation}\label{eq1.4}
\phi\left(x_0\right)=\frac{\int_{B_e\left(x_0\right)}\left|\langle \nabla \phi\left(x_0\right) ,\left(x-x_0\right)\rangle \right|^{p-2} \phi(x) d x}{\int_{B_e\left(x_0\right)}\left|\langle \nabla \phi\left(x_0\right) ,\left(x-x_0\right)\rangle \right|^{p-2} d x}-\frac{p}{N+p} \frac{\epsilon^2}{2} \Delta_p^G \phi\left(x_0\right)+o\left(\epsilon^2\right)
\end{equation}
where $ \Delta_p^G u = \frac{1}{p}|\nabla u|^{2-p} \operatorname{div}\left(|\nabla u|^{p-2} \nabla u\right) $ is a variant of the p-Laplacian, known as the game p-Laplacian, or normalized p-Laplacian. This version of the p-Laplacian is related to a two player game called tug-of-war with noise, see \cite{Lewicka}.  One can also look at surface integrals, which result in expansion theorems of the form

\begin{equation}\label{eq1.5}
 \phi\left(x_0\right)=\frac{\int_{\partial B_\epsilon\left(x_0\right)}\left|\langle \nabla \phi\left(x_0\right) ,\left(x-x_0\right)\rangle \right|^{p-2} \phi(x) d x}{\int_{\partial B_e\left(x_0\right)}\left| \langle \nabla \phi\left(x_0\right) ,\left(x-x_0\right)\rangle \right|^{p-2} d x}-\frac{p}{N+p-2} \frac{\epsilon^2}{2} \Delta_p^G \phi\left(x_0\right)+o\left(\epsilon^2\right)
\end{equation}
\noindent These expansion theorems were shown in two dimensions in \cite{GiorgiSmits} and in all dimensions in \cite{Alawamleh}.  In a paper of one of the coauthors \cite{SmitsAveraging} a numerical scheme was developed to solve the Dirichlet problem 

\begin{equation}\label{eq1.6}
\begin{cases}-\Delta_p^G u= 0 & \text { in } \Omega \\ u=G & \text { on } \partial \Omega\end{cases}
\end{equation}

which relied on convergence, via the dynamic programming principle, of an averaging operator, locally adapted to the p-Laplacian, called the p-average. Using a different approach, \cite{Oberman} Oberman used a scheme based on a weighted interpolation between the mean and midrange to attack the same problem. All of these schemes were synthesized in a recent paper \cite{Manfredi} where a framework is developed for families of averages to capture an asymptotic mean value property for the game p-Laplacian.


\section{p-Averages and their properties}
The importance of central tendency has a long history, and includes the arithmetic mean, the mode, median, midrange and various weighted and interpolated versions of these.  To be an average or measure of central tendency certain properties are natural. In particular larger sets should have larger means, adding constants to all values in a set should shift the average and of course the central tendency is between the maximum and minimum.  To this end we follow \cite{Manfredi} where averages are defined as

\begin{definition}\label{def1:def1}
    We say that an operator $A: L^{\infty}(S) \rightarrow \mathbb{R}$ is an average if the following assumptions hold:
    
	  (1) (Stable) $\inf _{y \in S} \phi(y) \leq A[\phi](x) \leq \sup _{y \in S} \phi(y)$ for all $x \in \Omega$
	  
(2) (Monotone) If $\phi \leq \psi$ in $s$ then $A[\phi] \leq A[\psi]$ in $s$.

(3) (Affine invariance) $A[\lambda \phi+\xi]=\lambda A[\phi]+\xi$.

	\end{definition}

The main example studied in connection to the PDE literature is when $S = \Omega_E$ where $\Omega \subset \mathbb{R}^n$,  $\Omega_E$ is an extension of $\Omega$ by strip of width 1, the average is indexed by $\varepsilon >0$ and is localized over balls, namely

$$A[\phi](x) =\left(\frac{p-2}{p+n}\right)\left(\frac{\max _{B_{\varepsilon}(x)} \phi +\min _{B_{\varepsilon}(x)} \phi}{2}\right)+\left(\frac{2+n}{p+n}\right)\frac{1}{|{B_{\varepsilon}(x)}|} \int_{B_{\varepsilon}(x)} \phi(y) d y$$

In this paper we will use a different family of averages, following \cite{SmitsAveraging} called p-averages and which can be computed numerically using a Newton bracketing scheme.  We define the p-average for a set, $1 < p \leq \infty$, first considered for finite sets and their limits in \cite{SmitsAveraging} and more recently, in general by \cite{Ishiwata}.   

    \begin{definition}\label{def:def2}
 
        Given a continuous function $\phi$ on a compact topological space $X$ equipped with a positive Radon measure $\nu$ there are unique real values $A_p^X(\phi)$ that solve the variational problem
$$
\left\|\phi-A_p^X(\phi)\right\|_{L^p(X, \nu)}=\min _{\lambda \in \mathbb{R}}\|\phi-\lambda\|_{L^p(X, \nu)},
$$

    \end{definition}

$A_p^X$  has the algebraic properties making it into an average as in Definition 2.1. A similar definition holds for $p=1$, except if $\nu$ has point masses the 1-average, known as the median need not be unique.  This is the classical situation for counting measure on an even number of points, where the median fails to be unique because $L^1$ is not strictly convex. The $A_p$ average also has some analytic properties, see e.g. \cite{Ishiwata}.

\vspace{0.2in}
 
\begin{itemize}
    \item  for $1 \leq p<\infty, A_p^X(\phi)$ is characterized by the equation
$$
\int_X\left|\phi(y)-A_p^X(\phi)\right|^{p-2}\left[\phi(y)-A_p^X(\phi)\right] d \nu=0
$$
    
    \item Let $\phi(y)$ and $\psi(y)$ be two continuous functions with $\left\|\phi-\psi\right\|_{L^{\infty}(X), \nu} \leq \delta$ then
$$
A_{p}^X(\phi)-\delta \leq A_{p}^X(\psi) \leq A_{p}^X(\phi)+\delta
$$
 In fact 
 
 \begin{equation}
     \left|\left\|\phi-A_p^X(\phi)\right\|_{L^p(X)}-\left\|\psi-A_p^X(\psi)\right\|_{L^p(X)}\right| \leq\|\phi-\psi\|_{L^p(X)}
 \end{equation}
 
\item If
$ \phi \leq \psi \text { then } A_p^X(\phi) \leq A_p^X(\psi)$, if strict equality occurs on a set of positive measure and $1<p<\infty$ there is strict inequality, namely $A_p^X(\phi) < A_p^X(\psi)$.
\end{itemize}

 \vspace{0.2in}

If $\nu$ is a probability measure, there is a natural measure of dispersion similar to standard deviation $\sigma_p(\phi)= \left\|\phi-A_p^X(\phi)\right\|_{L^p(X, \nu)}$ occurs and the last equation can be written
 \begin{equation}
     \left|\sigma_p(\phi)- \sigma_p(\psi) \right| \leq\|\phi-\psi\|_{L^p(X)}
 \end{equation}
 
 \vspace{0.2in}

Averaging operators over balls in $D \subset \mathbb{R}^n$ are smoothing as we next describe. Let $\mathcal{D} \subset \mathbb{R^N}$ be open, bounded and connected. For $\epsilon \in(0,1)$ fixed, define the thickened inner boundary $\Gamma_{\epsilon} = \{x\in \mathcal{D}; dist(x,\partial \mathcal{D})\leq \epsilon\}$ of $\mathcal{D}$,  $\Gamma_{out}= \{x\in \mathbb{R}\backslash \mathcal{D}; dist(x,\partial \mathcal{D})\leq 1\}$ the outer boundary and the closed domain $\mathcal{D}^\diamond =  \mathcal{D} { \displaystyle \cup } \Gamma_{out}.$ For $0<\epsilon <1$ and $x \in D$ we define $A_\epsilon(u(x))=A_p^{B_\epsilon (x)}(u)$.
 
 \newpage

\begin{theorem}
 Let $\epsilon >0$, ${A}_\epsilon (u(x))$ is H\"{o}lder continuous for $p> 2$ when $u$ is a bounded function.  
\end{theorem}

\begin{proof}
 Let $x_1, x_2 \in \mathcal{D}$ and $\epsilon >0$. By affine invariance, we can suppose that ${A}_\epsilon(u(x_1))=0$ and ${A}_\epsilon (u(x_2))=\lambda >0$. Using the characterization of ${A}_\epsilon (u(x))$ we have

\vspace{0.2in}

$\left|\int_{B_\epsilon (x_1)}\left|u(x) \right|^{p-2}u(x)\,dx-\int_{B_\epsilon (x_2)}\left|u(x) \right|^{p-2}u(x)\,dx\right|$

$=\left|\int_{B_\epsilon (x_2)}\left|u(x) -\lambda \right|^{p-2} \left(u(x)-\lambda\right)\,dx-\int_{B_\epsilon (x_2)}\left|u(x) \right|^{p-2}u(x)\,dx\right|$\\

We bound the first side of the equality above by a set difference and the second side of the equality below using a classical inequality \cite{Notes}.   

\vspace{0.1in}

$\left|\int_{B_\epsilon (x_1)}\left|u(x) \right|^{p-2}u(x)\,dx-\int_{B_\epsilon (x_2)}\left|u(x) \right|^{p-2}u(x)\,dx\right|\leq ||u||_{L^{\infty}}^{p-1}|B_\epsilon (x_1)\triangle B_\epsilon (x_2)|$\\

the right side of the equality is first factored as

$\left|\int_{B_\epsilon (x_2)}\left|u(x) -\lambda \right|^{p-2} \left(u(x)-\lambda\right)\,dx-\int_{B_\epsilon (x_2)}\left|u(x) \right|^{p-2}u(x)\,dx\right|$

$=|\lambda|^{p-1}\left|\int_{B_\epsilon (x_2)}\left|\frac{u(x)}{\lambda} -1 \right|^{p-2} \left(\frac{u(x)}{\lambda}-1\right)\,dx-\int_{B_\epsilon (x_2)}\left|\frac{u(x)}{\lambda} \right|^{p-2}\frac{u(x)}{\lambda}\,dx\right|.$\\

 set $t(x)=\frac{u(x)}{\lambda}$ the last term becomes
 
$$=|\lambda|^{p-1}\left|\int_{B_\epsilon (x_2)}\left|t(x) -1 \right|^{p-2} \left(t(x)-1\right)\,dx-\int_{B_\epsilon (x_2)}\left|t(x) \right|^{p-2}t(x)\,dx\right|.$$

Now for any $a, b \in \mathbb{R}^n$ 
$\langle |b|^{p-2}b-|a|^{p-2}a,b-a\rangle \geq 2^{2-p}|b-a|^p$ if $p\geq 2$.\\ Thus $(|t(x)|^{p-2}t(x)-|t(x)-1|^{p-2})(t(x)-1) \geq \frac{1}{2^{p-2}}$ yields that:\\
$\left|\int_{B_\epsilon (x_2)}\left|u(x) -\lambda \right|^{p-2} \left(u(x)-\lambda\right)\,dx-\int_{B_\epsilon (x_2)}\left|u(x) \right|^{p-2}u(x)\,dx\right|$\\
$ \geq |\lambda|^{p-1} \frac{1}{2^{p-2}} (B_\epsilon (x_2))$

\vspace{0.5cm}
So we have\\
$$|\lambda|^{p-1} \frac{1}{2^{p-2}} |B_\epsilon (x_2)|\leq||u||_{L^{\infty}}^{p-1} |B_\epsilon (x_1)\triangle B_\epsilon (x_2)|$$ or $$\lambda \leq ||u||_{L^{\infty}}\left[ 2^{p-2}\frac{|B_\epsilon (x_1)\triangle B_\epsilon (x_2)|}{|B_\epsilon (x_2)|}\right]^{\frac{1}{p-1}}.$$

\vspace{0.2in}
Classical estimates, see \cite{Lewicka} for the set difference of balls give us  $|B_\epsilon (x_1)\triangle B_\epsilon (x_2)|=
|B_\epsilon (0)\triangle B_\epsilon (x_1-x_2)|
= \frac{2}{\epsilon} \frac{V_{N-1}}{V_{N}} |B_{\epsilon}(0)||x_1-x_2|$ where $V_n=|B_1(0)|$ and $V_{n-1}$ is the volume of the unit ball in $\mathbb{R}^{n-1}.$
\vspace{0.1in}
Thus $|\lambda| \leq C ||u||_{L_{\infty}} |x_1-x_2|^{\frac{1}{p-1}}$ where $C=2^{\frac{p-2}{p-1}} \left( \frac{2}{\epsilon} \frac{V_{N-1}}{V_{N}} \frac{|B_{\epsilon}(0)|}{|B_{\epsilon}(x_2)|} \right)^{\frac{1}{p-1}}$ and hence
$$\left|{A}_\epsilon (u(x_1))-{A}_\epsilon (u(x_2))\right| \leq C ||u||_{L^{\infty}} |x_1-x_2|^{\frac{1}{p-1}}.$$

\end{proof}

At this point we observe the special cases $p=2$ and $p=\infty$ are simple to understand amongst the p-averages and hence used most frequently in PDE approaches to stochastic games like Tug-of-War \cite{Lewicka}.  As mentioned, there are cases when $p=1$, the set $X$ is discrete with an even number of points and there is no longer uniqueness for the minimizer. On the other hand, there is uniqueness for all values $p>1$ and one can study the behavior as $p \rightarrow 1$.  We quote the following result from the first author's PhD thesis.

\begin{theorem}
Given data $\{x_1, x_2,..., x_k, x_{k+1},..., x_{2k}\}$, with $x_1< x_2<...< x_k< x_{k+1}<...< x_{2k}$ and $ c_p=\argmin_{c \in \mathbb{R}} \displaystyle\sum_{i=1}^{2k}\left|x_i-c \right|^{p}$, then $$p(c)=\displaystyle \prod_{x_k \leq c,i=1, \cdots,k}(c-x_i)-\displaystyle \prod_{x_{k=1} \geq c,i=k+1, \cdots,2k}(x_i-c)=0$$ has exactly one solution $x_k < c_{+1} < x_{k+1}$, called the $\gamma$-median  and $c_{+1}=\lim_{p\to 1+} c_p.$  
\end{theorem}

There is another p-average that is situated directly between the 2-average and $\infty$-average, namely the $4$-average. If $u$ is a smooth function with non-vanishing gradient  we defined the game p-Laplacian as

\begin{equation}
    \Delta_p^G u:=\frac{1}{p}|\nabla u|^{2-p} \operatorname{div}\left(|\nabla u|^{p-2} \nabla u\right)
\end{equation}
By expanding the derivatives, one obtains

\begin{equation}
    \Delta_p^G u=\frac{1}{p} \Delta_2 u+\frac{p-2}{p}|\nabla u|^{-2} \sum_{i, j} \frac{\partial u}{\partial x_i} \frac{\partial u}{\partial x_j} \frac{\partial^2 u}{\partial x_i \partial x_j}
\end{equation}
the second operator on the right is the definition of the game $\infty$-Laplacian, that is 

\begin{equation}
    \Delta_{\infty}^G u:=|\nabla u|^{-2} \sum_{i, j} \frac{\partial u}{\partial x_i} \frac{\partial u}{\partial x_j} \frac{\partial^2 u}{\partial x_i \partial x_j}
\end{equation}

and the first term on the right is $ \Delta_2^G u=\frac{1}{2} \Delta_2 u$. So that, $\Delta_p^G u=\frac{2}{p} \Delta_2^G u+ \frac{p-2}{p}\Delta_{\infty}^G u.  $

One can define 
\begin{equation}
    \Delta_1^G u= 2 \Delta_2^G u - \Delta_{\infty}^G u.  
\end{equation}
and recover the relationship
\begin{equation}
    \Delta_p^G u=\frac{1}{p} \Delta_1^G u+ \frac{1}{q}\Delta_{\infty}^G u.  
\end{equation}

When $p=4$ we have an intermediary between the 2 and $\infty$ game Laplacian, that is
\begin{equation}
    \Delta_4^G u=\frac{1}{2} \Delta_2^G u+ \frac{1}{2}\Delta_{\infty}^G u.  
\end{equation}

Formally,  $\Delta_p^G u$ can be seen as a singular, quasi-linear operator with domain a nonzero vector and a symmetric matrix, representing the gradient and Hessian matrix for u.  The operator is $F:\left(\mathbb{R}^n \backslash\{0\}\right) \times \mathcal{S}^n \rightarrow \mathbb{R}$, where
$$
F(\xi, A)=\operatorname{tr}[M(\xi) A],
$$
with
$$
M(\xi)=\frac{1}{p} I+\left(1-\frac{2}{p}\right) \frac{\xi \otimes \xi}{|\xi|^2}
$$
for $\xi \in \mathbb{R}^n \backslash\{0\}$.  That is with $\xi=\nabla u(x) \text { and } A=\nabla^2 u(x) = D^2u(x)$ one has

\begin{equation}\label{eqn2.9}
\Delta_p^G u(x) = \frac{1}{p}\operatorname{tr}(A)+\frac{p-2}{p} \frac{\langle A \xi, \xi\rangle}{|\xi|^2}.
\end{equation}
\vspace{0.2cm}
\subsection{A statistical description of the 4-average}
While p-averages for $p=1,2, \infty$ have known descriptive solutions, is seems largely unknown that  $A_4^X(\phi)$ has an explicit solution with a statistical interpretation.

\begin{lemma}\label{lem2.7} If $\nu$ is a probability measure, then $A_4^X(\phi)$, the unique minimizer for the polynomial $$p(x) =\int_X\left|\phi(y)-x\right|^{4-2}\left[\phi(y)-x\right] d \nu =\int_X\left(\phi(y)-x\right)^{3} d \nu=0$$  can be described as in terms of the mean, variance, and skewness of $\phi$, in particular  

$$A_4^X(\phi) =\frac{\sigma}{\sqrt[3]{2}}\left(\sqrt[3]{\kappa+\sqrt[2]{\kappa^2+4}}+\sqrt[3]{\kappa-\sqrt[2]{\kappa^2+4}}\right) + \bar{t}$$

Where $\bar{t}$ is the mean, $\sigma$ is the variance, and $\kappa$ is the skewness of $\phi.  $
\end{lemma}

\begin{proof}
\vspace{0.1in}
To describe the minimizer algebraically we expand the cubic equation, getting 
 
 \begin{equation}\label{eq:cubic}
     x^3-3x^2 \int_X \phi(y) d \nu +3x \int_X \phi(y)^2 d \nu  - \int_X \phi(y)^3 d \nu=0
 \end{equation}

We follow the classical solution of Cardano and set $a=1, b= -3 \int_X \phi(y) d \nu, c=3\int_X \phi(y)^2 d \nu, d=-\int_X \phi(y)^3 d \nu$ and depress the cubic equation~\eqref{eq:cubic} meaning rewrite it as $ t^3+pt+q=0 $ using the change of variable $x=t-\frac{b}{3a}=t+\bar{t}$.
$$
\begin{aligned}
&p=\frac{3 a c-b^2}{3 a^2} \\
&q=\frac{2 b^3-9 a b c+27 a^2 d}{27 a^3}
\end{aligned}
$$

In terms of classical statistical quantities $$p=\frac{9\int_X \phi(y)^2 d \nu- (-3 \int_X \phi(y) d \nu)^2}{3} =3\sigma^2$$

while

\begin{align*}
 q &=\frac{2 b^{3}-9 a b c+27 a^{2} d}{27 a^{3}}&\\&=\frac{2(-3 \int_X \phi(y) d \nu)^3-9(-3 \int_X \phi(y) d \nu)3\int_X \phi(y)^2 d \nu+27(-\int_X \phi(y)^3 d \nu}{27}&\\ &=-2\bar{t}^3+3\bar{t}(\sigma^2+\bar{t}^2)-\frac{\sum_{i}t_i^3}{1}&\\&=\bar{t}^3+3\bar{t}\sigma^2-\int_X \phi(y)^3 d \nu&\\&=-\sigma^3(\frac{{\int_X \phi(y)^3 d \nu}-3\bar{t}\sigma^2-\bar{t}^3}{\sigma^3})&\\&=-\sigma^3 \kappa\text{ where $\kappa$ is the skewness.}&\\  
\end{align*}
Since the discriminant $\Delta =-(4p^3+27q^2)<0$, $t^3+pt+q=0$ has one real root and two non-real complex conjugate root with the real root given by $$t=\sqrt[3]{\frac{-q}{2}+\sqrt[2]{\frac{q^2}{4}+\frac{p^3}{27}}}+\sqrt[3]{\frac{-q}{2}-\sqrt[2]{\frac{q^2}{4}+\frac{p^3}{27}}}$$

Recalling that $x=t+\bar{t}$ and substituting in the statistical meanings for p and q, we have $$t=\sqrt[3]{\frac{\sigma^3 \kappa}{2}+\sqrt[2]{\frac{(-\sigma^3 \kappa)^2}{4}+\frac{(3\sigma^2)^3}{27}}}+\sqrt[3]{\frac{\sigma^3 \kappa}{2}-\sqrt[2]{\frac{(-\sigma^3 \kappa)^2}{4}+\frac{(3\sigma^2)^3}{27}}}$$
and finally $$x=\sqrt[3]{\frac{\sigma^3 \kappa}{2}+\sqrt[2]{\frac{(-\sigma^3 \kappa)^2}{4}+\frac{(3\sigma^2)^3}{27}}}+\sqrt[3]{\frac{\sigma^3 \kappa}{2}-\sqrt[2]{\frac{(-\sigma^3 \kappa)^2}{4}+\frac{(3\sigma^2)^3}{27}}} + \bar{t}$$ which can be simplified as $$A_4^X(\phi) = x=\frac{\sigma}{\sqrt[3]{2}}\left(\sqrt[3]{\kappa+\sqrt[2]{\kappa^2+4}}+\sqrt[3]{\kappa-\sqrt[2]{\kappa^2+4}}\right) + \bar{t}.$$
\end{proof}

\subsection{The uniqueness of p=4}
 One might hope that other p-averages have explicit solutions, especially p=6. From algebra we know that a general quintic equation is not solvable, however the 6-average solves a special class of quintic equations which might be solvable. We show this is not the case. Consider the random data set X=\{1,6,11,13,19\} with the counting measure. 
 $A_6^X = x$ solves the equation, $$6f(x) = h'(x) = 6(x-1)^5 + 6(x-6)^5 + 6(x-11)^5 + 6(x-13)^5 + 6(x-19)^5=0.$$

Expanding this out and depressing by the average, that is, using the substitution x = y+10, leads to the depressed quintic polynomial $$
p(x)=x^5+376 x^3+72 x^2+13460 x+156.
$$

The following theorem from Dummit\cite{Dummit} decides when $p(x)$ is solvable by radicals.

\begin{theorem}\label{thm:thm2.3}
The irreducible quintic $p(x)=x^5+px^3+qx^2+rx+s \in  \mathbf{Q}[X]$ is solvable by radicals if and only if the polynomial $p_{20}(x)$ has a rational root where $p_{20}(x)$ is the resolvent sextic polynomial corresponding to $p(x)$.  
\end{theorem}

In the above case, the resolvent sextic polynomial is 
\begin{flalign*}
p_{20}(x)&\\&=x^6+107680x^5-4167324992x^4-633810584502272x^3&\\&-633810584502272x^2-2545206831640273748008x&\\&+7102938318637196554440048.    
\end{flalign*}

Checking all the factors of 7102938318637196554440048 with a computer algebra system show there is no rational root.  

\vspace{0.5in}
\section{Asymptotic Mean Value Properties}

Having discussed various properties of p-averaging, we now employ them to study the game p-Laplacian.  We first quote the following definition from \cite{ManParRoss} and its importance in numerical schemes.  Here we use $S = \Omega_E$, the extension of a domain $\Omega$.
	\begin{definition}\label{def2:def2}
		We say that family of averages $\left\{A_{\varepsilon}\right\}_{\varepsilon>0}$ is an \textbf{(asymptotic) mean value property AMVP} for the p-Laplacian if for every function $\phi \in C_b^\infty(\Omega_E)$ with $\nabla \phi \neq 0$ we have
		\begin{equation}\phi(x)=A_{\varepsilon}[\phi](x)+c_{p,n} \varepsilon^{2}\left(-\Delta_{p}^{G} \phi(x)\right)+o\left(\varepsilon^{2}\right)\end{equation}
	\end{definition}
	where the constant in $o\left(\varepsilon^{2}\right)$ can be taken uniformly for all $x \in \Omega_E$
The usefulness of this definition is that the operator defined via
\begin{equation}S(\varepsilon, x, \phi(x), \phi)=\left\{\begin{array}{ccc}
\frac{1}{c_{p, n} \varepsilon^2}\left(\phi(x)-A_{\varepsilon}[\phi](x)\right) & \text { if } & x \in \Omega \\
\phi(x)-G(x) & \text { if } & x \in 
\text{ extended strip }
\end{array}\right.
\end{equation}

leads to a fixed point problem and the scheme 

\begin{equation}
    S\left(\varepsilon, x, u_{\varepsilon}(x), u_{\varepsilon}\right)=0 \quad \text { for all } \quad x \in \Omega_E \text {. }
\end{equation}

\noindent satisfies the needed stability, monotonicity and consistency conditions in Barles and Souganidis fundamental paper \cite{Souganidis} leading to convergence of a set of functions converging to the solution to (1.6).

\subsection{Previous results for continuous averages}

In the context of numerical schemes, the AMVP was essentially shown in 2 dimensions in \cite{Oberman} for mixed averages on the boundary of balls,  $\partial B_{\varepsilon}(x)$ $$A[\phi](x) =\alpha \left(\frac{\max _{\partial B_{\varepsilon}(x)} \phi +\min _{\partial B_{\varepsilon}(x)} \phi}{2}\right)+\beta \frac{1}{|{\partial B_{\varepsilon}(x)}|} \int_{\partial B_{\varepsilon}(x)} \phi(y) d y$$
 with $c_{p,2} = \frac{1}{2}$.  In Falcone et al. \cite{SmitsAveraging}, for $p>2$ and in 2 dimensions, the set of averages $A_p^X$ where $X=\partial B_{\varepsilon}(x)$ was essentially shown to have the AMVP with $c_{p,2} = \frac{1}{2}$ as well.
 
 \vspace{0.1in}
 As previously mentioned, in \cite{Manfredi} the tug-of-war inspired averages
 $$A[\phi](x) =\left(\frac{p-2}{p+n}\right)\left(\frac{\max _{B_{\varepsilon}(x)} \phi +\min _{B_{\varepsilon}(x)} \phi}{2}\right)+\left(\frac{2+n}{p+n}\right)\frac{1}{|{B_{\varepsilon}(x)}|} \int_{B_{\varepsilon}(x)} \phi(y) d y$$ produce an AMVP with $c_{p,n} = \frac{p}{2(n+p)}$. 
 
 \vspace{0.1in}
 Some authors use the title, normalized p-Laplacian and may not include a factor $\frac{1}{p}$ in their definitions which will change the constants $c_{p,n}$ accordingly.  In \cite{Ishiwata}, with $X=\overline{B_{\varepsilon}(x)}$ the authors derive asymptotic formulas in all dimensions for smooth functions with non-vanishing gradients, 
 
 $$A_p^X(v) =v(x)+\frac{1}{2} \frac{p}{n+p}\Delta_p^G v(x) \varepsilon^2+o\left(\varepsilon^2\right)$$ as $\varepsilon \rightarrow 0$. This yields an AMVP for p-averages on solid balls with
 
 \begin{equation}
     c_{p,n} =\frac{p}{2(n+p)}.
 \end{equation}

 In the same article the authors study a family of p-averages with $X=\partial B_{\varepsilon}(x)$, and find an asymptotic mean value principle on spheres with
 
 \begin{equation}
     c_{p,n}=\frac{p}{2(n+p-2)}. 
 \end{equation}

  When $n=2$ this reduces to the same $c_{p,2}=\frac{1}{2}$ for all p.

 \vspace{0.2in}
 \textbf{Remark 1} When context is clear and $X$ depends on $\varepsilon$, we will follow the notation in definition 3.1 and write $A_{\varepsilon}^p$ in place of $A_p^{X}$.
 
 \vspace{0.2in} 
 
 \textbf{Remark 2} A common thread among all previous results was that a numerical scheme for and $p \neq 2$ would necessarily involve arbitrarily small directional resolution in addition to small spatial resolution.  This is because sampling over a ball or sphere requires vectors in all possible directions. In the next sections, we find the first nontrivial examples where the directional resolution can itself be discrete and still yield an asymptotic mean value principle.  

\subsection{Two dimensional discrete results}
\vspace{0.1in}

In two dimensions it is clearer to formulate the directions in using the complex exponential or trigonometric functions. For even integer value $p$, say $p=2k$ we define a family of averages indexed by $\epsilon >0$ by describing the set with 2k+2 equally spaced vertices on the boundary of the ball as $S = S_{\epsilon} = \left\{  x+\epsilon (cos(\frac{2 \pi j}{2k+2}),sin(\frac{2 \pi j}{2k+2})), j=0, \ldots, 2k+1 \right\}$. One can also think of this as a set of vertices on a polygon in the complex plane. We then define the p-averaging as $ A_{2k}^{S} = A_{\varepsilon}^{2k} $.  \\

We simplify matters, set $\xi_j =\epsilon \widetilde{\eta_{j}} = \epsilon (cos(\widetilde{\theta_j)},sin(\widetilde{\theta_j)}) = \epsilon (cos(\frac{2 \pi j}{2k+2}),sin(\frac{2 \pi j}{2k+2})) $ so that  $S = S_{\epsilon} =  \left\{\left(x+ \xi_j \right), j=0, \ldots, 2k+1 \right\}$. We follow the proof in \cite{SmitsAveraging} for two dimensions which is useful for such applications as inpainting and image processing. Throughout we assume $\varepsilon > 0$ is small enough so that with $x \in \Omega$, $B(x,\varepsilon) \subset \Omega_E$.  
Before proceeding to our theorems for the p-Laplacian in the plane, we give a motivation for why such theorems might occur. The mean value theorem for all balls and all harmonic functions is known in complex variables, but a lesser known result is the following due to \cite{Walsh}:

\vspace{0.2in}
\begin{theorem}\label{thm:complexpoly}
Any polynomial $p(z)$ of degree $n$ over the complex plane, satisfies a mean value property on any regular polygon with $2n$ or more vertices, namely if $\theta$ is an arbitrary angle and we set  $z_{j}=z+r e^{i \theta} e^{\frac{2 \pi i}{2 n} j}=z+r e^{\theta i} e^{\frac{\pi i}{n} j}$ where $j=0,1,2, \ldots, 2 n-1$, then $p(z)=\frac{\sum_{j=0}^{2 n-1} p\left(z_{j}\right)}{2 n}$. Here $z$ is the center of the polygon.
 \end{theorem}
 We repeat the simple proof which relies only on geometric series and the binomial theorem.   
 \begin{proof}
    
    It's enough to show termwise that $\sum_{j=0}^{2n-1}(z+re^{\theta i}e^{\frac{\pi}{n}ij})^m=2n z^m$ holds for any $m=1,2,\dots,n$ since it holds trivially for $m=0$.

  \begin{flalign*}
  \sum_{j=0}^{2n-1}(z+re^{\theta i}e^{\frac{2\pi}{2n}ij})^m
  &\\&=\sum_{j=0}^{2n-1} \sum_{k=0}^{m}\binom{m}{k}   z^{m-k}r^k e^{\theta i k}e^{\frac{\pi}{n}ijk}&\\&=\sum_{k=0}^{m} \sum_{j=0}^{2n-1}\binom{m}{k}   z^{m-k}r^k e^{\theta i k}e^{\frac{\pi}{n}ijk}&\\
  \end{flalign*}

so that
  \begin{flalign*}
  \sum_{j=0}^{2n-1}(z+re^{\theta i}e^{\frac{2\pi}{2n}ij})^m
  &\\&=2 n z^{m}+m z^{m-1}r e^{\theta i } \frac{1-e^{2\pi i}}{1-e^{\frac{\pi}{n}i}}+\dots+&\\&\binom{m}{m-1}z r^{m-1} e^{(m-1) i \theta  }\frac{1-e^{2(m-1)\pi i}}{1-e^{\frac{(m-1)\pi}{n}i}}+r^m e^{m i \theta}\frac{1-e^{2 m\pi i}}{1-e^{\frac{m \pi}{n}i}}&\\
  \end{flalign*}

  For $t$ an even integer, $1-e^{t\pi i}= 0$. Thus, $1-e^{2\pi i}=1-e^{4\pi i}=\dots =1-e^{2(m-1)\pi i}=1-e^{2m\pi i}=0$ while none of the denominators $1-e^{\frac{ \pi}{n}i},1-e^{\frac{2 \pi}{n}i},\dots ,1-e^{\frac{(m-1) \pi}{n}i}$ are zero as $0 < m \leq n.$ So all the non-leading terms are zero implying 
     $$\sum_{j=0}^{2n-1}(z+re^{\theta i}e^{\frac{2\pi}{2n}ij})^m=2n z^m \text{ for } m=1,2,\dots ,n \text{ }$$
 
\end{proof}

\newpage
\begin{theorem}\label{thm:Theorem2.6}

$ The family \left\{A_{\varepsilon}^{4}\right\}_{\varepsilon >0}$ has the asymptotic mean value property for the 4-Laplacian with $c_{4,2} = \frac{1}{2}.$ 
\end{theorem}
\begin{proof}

\vspace{0.1in}
Some manipulations of the  $\argmin$ using the translation part of affine invariance of the p-average shows

$$ A_{\varepsilon}^{4}[\phi](x) - \phi(x) = \argmin_{c \in \mathbb{R}} \sum_{j=0}^{5}\left|(\phi(x + \xi_j)-\phi(x))-c\right|^{4}$$

\vspace{0.1in}
and by the scaling part of affine invariance

$$ \frac{A_{\varepsilon}^{4}[\phi](x) - \phi(x)}{\frac{\varepsilon^2}{2}}= \argmin_{d \in \mathbb{R}} \sum_{j=0}^{5}\left|(\phi(x + \xi_j)-\phi(x))-\frac{\varepsilon^2}{2} d\right|^{4}$$

Now employing Taylor's theorem $$\phi(x + \xi_j) - \phi(x) = \langle \varepsilon \nabla \phi(x) , \widetilde{\eta_j} \rangle + \frac{\varepsilon^2}{2} \langle D^2\phi(x)\widetilde{\eta_j} , \widetilde{\eta_j} \rangle + o(\varepsilon^2)$$

By continuity of $A_{\epsilon}^4$, in the limit one could replace  $\phi$ with the quadratic function up to second order derivatives, that is in the limit

$$
\frac{A_{\varepsilon}^{4}[\phi](x) - \phi(x)}{\frac{\varepsilon^2}{2}} = \argmin_{d \in \mathbb{R}} \sum_{j=0}^{5}\left|    \langle \varepsilon \nabla \phi(x) , \widetilde{\eta_j} \rangle + \frac{\varepsilon^2}{2} \langle D^2\phi(x)\widetilde{\eta_j} , \widetilde{\eta_j} \rangle    -\frac{\varepsilon^2}{2} d  \right|^{4}
$$

 Factoring out an $\varepsilon$ from the $\argmin$,
 
 $$
 \frac{A_{\varepsilon}^{4}[\phi](x) - \phi(x)}{\frac{\varepsilon^2}{2}} = \argmin_{d \in \mathbb{R}} \sum_{j=0}^{5}\left|\langle \nabla \phi(x) , \widetilde{\eta_j} \rangle + \frac{\varepsilon}{2} \langle D^2\phi(x)\widetilde{\eta_j} , \widetilde{\eta_j} \rangle-\frac{\varepsilon}{2} d \right|^{4}
 $$
 
 \noindent In the definition of AMVP,  $\nabla \phi (x) \neq 0$  so we  assume it is in the positive x direction, $$\nabla\phi (x) = <|\nabla \phi (x)|, 0 > .$$ With this assumption, the angles may change, taking the form, for some $0 < \eta < \frac{\pi}{3}$ of  $\theta_j = \widetilde{\theta_j} + \eta$
 $$
 \Delta_{\infty}^{G} \phi (x)=\partial_{11}\phi(x)  \text{ and } \Delta_{1}^{G} \phi (x)=\partial_{22}\phi(x) .$$
Then
\vspace{0.1in}

$\langle D^2\phi(x)\widetilde{\eta_j} , \widetilde{\eta_j} \rangle  =  \Delta_{\infty}^{G} \phi(x) \cos ^{2} \theta_{j}+ \Delta_{1}^{G} \phi(x) \sin ^{2} \theta_{j}+ 2 \partial_{12} \phi(x) \sin \theta_{j} \cos \theta_{j} $
 
 \vspace{0.1in}
 and so

\vspace{0.1in}
 $
 \frac{A_{\varepsilon}^{4}[\phi](x) - \phi(x)}{\frac{\varepsilon^2}{2}} = \argmin_{d \in \mathbb{R}} \sum_{j=0}^{5} \bigg \vert ( \langle \nabla \phi(x) , \theta_j \rangle+ \frac{\varepsilon}{2} ( \Delta_{\infty}^{G} \phi(x) \cos ^{2} \theta_{j}$ +
 
 $ \Delta_{1}^{G} \phi(x) \sin ^{2} \theta_{j} + 2 \partial_{12} \phi(x) \sin \theta_{j} \cos \theta_{j})-\frac{\varepsilon}{2} d  \bigg \vert ^{4}
 $

 \vspace{0.1in}

 Using $\sin{2 \theta} = 2 \sin{\theta} \cos{\theta}$, $\Delta^{G}_4 = \frac{1}{4} \Delta^{G}_1 + \frac{3}{4} \Delta^{G}_{\infty}$ and $\sin^2{\theta} = 1 - \cos^2{\theta}$ we get
 
 \vspace{0.2in}
 $\frac{A_{\varepsilon}^{4}[\phi](x) - \phi(x)}{\frac{\varepsilon^2}{2}} =  \argmin_{d \in \mathbb{R}} \sum_{j=0}^{5} \bigg \vert  
 |\nabla \phi(x)|\cos{\theta_j} + \frac{\varepsilon}{2}  \partial_{12} \phi(x)  \sin{2 \theta_j}  +  \frac{\varepsilon}{2} ( \Delta_{\infty}^{G} \phi(x) - \Delta_{1}^{G} \phi(x) )(\cos ^{2} \theta_{j} - \frac{3}{4} ) + \frac{\varepsilon}{2} \Delta_{4}^{G} \phi(x) -\frac{\varepsilon d}{2}  \bigg \vert ^4$

 \vspace{0.1in}
 substituting $t = d - \Delta_{4}^{G} \phi(x) $ and using the affine invariance we get 
 
 \vspace{0.1in}

 $\frac{A_{\varepsilon}^{4}[\phi](x) - \phi(x)}{\frac{\varepsilon^2}{2}} =  \Delta_{4}^{G} \phi(x)  + \argmin_{t \in \mathbb{R}} \sum_{j=0}^{5} \bigg \vert  
 |\nabla \phi(x)|\cos{\theta_j} + \frac{\varepsilon}{2}  \partial_{12} \phi(x) \sin{2 \theta_j}  +  \frac{\varepsilon}{2} ( \Delta_{\infty}^{G} \phi(x) - \Delta_{1}^{G} \phi(x) )(\cos ^{2} \theta_{j} - \frac{3}{4} )  -\frac{\varepsilon t}{2}  \bigg \vert ^4$
 
 \vspace{0.1in}
 call $a =  \Delta_{\infty}^{G} \phi(x) - \Delta_{1}^{G} \phi(x) $ and $b =  \partial_{12} \phi(x)$ this becomes 
 
 \vspace{0.2in}

 $\frac{A_{\varepsilon}^{4}[\phi](x) - \phi(x)}{\frac{\varepsilon^2}{2}} =  \Delta_{4}^{G} \phi(x)$ +

 $$\argmin_{t \in \mathbb{R}} \sum_{j=0}^{5} \bigg \vert  
 |\nabla \phi(x)|\cos{\theta_j} + \frac{\varepsilon}{2}  b  \sin{2 \theta_j}  +  \frac{\varepsilon}{2} a (\cos ^{2} \theta_{j} - \frac{3}{4} )  -\frac{\varepsilon t}{2}  \bigg \vert ^4$$

For $\varepsilon > 0$ fixed,  the $\argmin$ occurs at the unique place when the derivative is zero, that is the unique value $t = t(\varepsilon)$ with 

$$
0 = \frac{0}{4 } =  \sum_{j=0}^{5} \bigg( |\nabla \phi(x)|\cos{\theta_j} + \frac{\varepsilon}{2}  b  \sin{2 \theta_j}  +  \frac{\varepsilon}{2} a (\cos ^{2} \theta_{j} - \frac{3}{4} )  -\frac{\varepsilon t}{2}  \bigg)^3 
$$
 expanding this out, and keeping only the terms up to order $\varepsilon$ we get 
 
 $$
 0 = \sum_{j=0}^{5} |\nabla \phi(x)|^ 3\cos^3{\theta_j} + 3|\nabla \phi(x)|^2 \cos^2{\theta_j} \cdot \frac{\varepsilon}{2} \big( b  \sin{2 \theta_j} + a (\cos ^{2} \theta_{j} - \frac{3}{4} ) -t \big) 
 $$
 Recall that for $j=0  \ldots 5 ,$ we have $\theta_j= \widetilde{\theta_j} + \eta = \frac{j \pi}{3} + \eta$  where for this set $\cos{ \widetilde{\theta_j}}=\{1,\frac{1}{2},-\frac{1}{2},-1,-\frac{1}{2},\frac{1}{2}\}$ and $\sin{ \widetilde{\theta_j}}=\{0,\frac{\sqrt{3}}{2},\frac{\sqrt{3}}{2} ,0,-\frac{\sqrt{3}}{2}, -\frac{\sqrt{3}}{2}\} $. \\
 
 \noindent
 Basic trigonometric identities  show for any $\eta$, 
 
\begin{align*}
   & \sum \cos^3(\theta_j) \sin(\theta_j) = 0 \\
   &\sum \cos^4(\theta_j) = \frac{18}{8} = \frac{3}{4} \cdot \frac{6}{2} = \frac{3}{4} \sum \cos^2{\theta_j}\\ &\text{ and } \sum \cos^3(\theta_j)=0 
\end{align*}which implies that up to $o(\varepsilon)$,  $t=0$ and hence $A^4_{\varepsilon} $ has the asymptotic mean value property.  

 \end{proof}  
  
  \vspace{0.2in}
 One of the interesting interpretations of this theorem is that for a fine equilateral triangular mesh covering a region $\Omega \subset \mathbb{R}^2$, the $4-Laplace$ equation can be solved numerically with a fixed directional resolution. This is because the triangular lattice is a tessellation of $\mathbb{R}^2$.  The triangular lattice can be generated algebraically as the Eisenstein integers $$z = a+b\omega \text{ where a and b are integers and } \omega = e^{i\frac{2\pi}{3}}.$$
  
A more complicated construction will be discussed for the 24-cell honeycomb in 4-dimensions. We will extend this discrete averaging over polygons to even values of p in two dimensions and to some higher dimensional analogs called polytopes P, in the spirit of \cite{Japanese}.  In that article, they let $P(k)$ be the $k$-dimensional skeleton for P and  $\mathcal{H}_{P(k)}$ denote the set of $P(k)$-harmonic functions. For certain polytopes, there are polynomial invariants for the finite reflection groups of the types $H_{3}, H_{4}$ and $F_{4}$ which explicitly allows one to determine the solution space of functions satisfying a 2-average mean value property (MVP) related to the icosahedron and dodecahedron in three dimensions and the 24-cell, 600-cell, and 120-cell in four dimensions.  We extend the previous theorem in the plane to all even positive integers.

  \begin{theorem}\label{thm:Theorem1}
$ A_{\varepsilon}^{2  k}$ has the asymptotic mean value property for the $p=2 k$-Laplacian  for $k \in \mathbb{N}$, and $c_{2k,2} = \frac{1}{2}.$ 
\end{theorem}
\begin{proof}
			We take a slightly different approach and follow \cite{Ishiwata} where the authors note in their Theorem 3.2 that,  by continuity of $p-averages$, one only needs to show the theorem for quadratic functions of the form 
			\begin{equation}
			 \phi(y)=\phi(x)+\langle \vec{a} ,(y-x) \rangle+\frac{1}{2}\langle A(y-x), y-x\rangle, y \in B_{\varepsilon}(x)   
			\end{equation}
			
			Where $A=D^2\phi(x)$ and $\vec{a} = \nabla \phi(x)$.
		    As before, by affine invariance
			
			$$ A_{\varepsilon}^{2k}[\phi](x) - \phi(x) = \argmin_{c \in \mathbb{R}} \sum_{j=0}^{2k+1}\left|(\phi(x + \xi_j)-\phi(x))-c\right|^{2k}$$
			
			\vspace{0.1in}
			and

			$$ \frac{A_{\varepsilon}^{2k}[\phi](x) - \phi(x)}{\frac{\varepsilon^2}{2}}= \argmin_{d \in \mathbb{R}} \sum_{j=0}^{2k+1}\left|(\phi(x + \xi_j)-\phi(x))-\frac{\varepsilon^2}{2} d\right|^{2k}$$

			Using the form of $\phi$
			$$
			\frac{A_{\varepsilon}^{2k}[\phi](x) - \phi(x)}{\frac{\varepsilon^2}{2}} = \argmin_{d \in \mathbb{R}} \sum_{j=0}^{2k+1}\left|    \langle \varepsilon \nabla \phi(x) , \widetilde{\eta_j}\rangle + \frac{\varepsilon^2}{2} \langle D^2\phi(x)\widetilde{\eta_j} , \widetilde{\eta_j} \rangle    -\frac{\varepsilon^2}{2} d  \right|^{2k}
			$$
			and factoring out an $\varepsilon$ from the $\argmin$,
			
			$$
			\frac{A_{\varepsilon}^{2k}[\phi](x) - \phi(x)}{\frac{\varepsilon^2}{2}} = \argmin_{d \in \mathbb{R}} \sum_{j=0}^{2k+1}\left|\langle \nabla \phi(x) , \widetilde{\eta_j} \rangle + \frac{\varepsilon}{2} \langle D^2\phi(x)\widetilde{\eta_j} , \widetilde{\eta_j} \rangle-\frac{\varepsilon}{2} d  \right|^{2k}.
			$$
			
			Substituting $\vec{a}$ and $A$ from before 
			
				$$
			\frac{A_{\varepsilon}^{2k}[\phi](x) - \phi(x)}{\frac{\varepsilon^2}{2}} = \argmin_{d \in \mathbb{R}} \sum_{j=0}^{2k+1}\left| \langle \vec{a} , \widetilde{\eta_j} \rangle + \frac{\varepsilon}{2} \langle A \widetilde{\eta_j} , \widetilde{\eta_j} \rangle -\frac{\varepsilon}{2} d  \right|^{2k}
			$$
			Next we choose a rotation, orthogonal matrix $Q$ that makes $\vec{a}=Q\left(|a| \vec{e}_{1}\right)$, and define directions $\eta_j$ via $Q \eta_j = \widetilde{\eta_j}$ so that 
			
	$$\frac{A_{\varepsilon}^{2k}[\phi](x) - \phi(x)}{\frac{\varepsilon^2}{2}} = \argmin_{d \in \mathbb{R}} \sum_{j=0}^{2k+1}\left|\langle |\vec{a}| Q(\vec{e}_1) , Q({\eta_j})\rangle + \frac{\varepsilon}{2} \langle A Q({\eta_j}) , Q({\eta_j}) \rangle -\frac{\varepsilon}{2} d  \right|^{2k}
			$$
			
			$$
			= \argmin_{d \in \mathbb{R}} \sum_{j=0}^{2k+1}\left|\langle |\vec{a}| \vec{e}_1 , {\eta_j} \rangle + \frac{\varepsilon}{2} \langle Q^{T} A Q({\eta_j}) , {\eta_j} \rangle-\frac{\varepsilon}{2} d  \right|^{2k}
			$$
			and taking derivatives we see that $d = d(\epsilon)$ satisfies
			
			$$
0=\sum_{j=0}^{2k+1}\left(\langle|\vec{a}| \vec{e}_1 , {\eta_j}\rangle + \frac{\varepsilon}{2} \langle Q^{T} A Q({\eta_j}) , {\eta_j}\rangle -\frac{\varepsilon}{2} d \right)^{2 k-1}.
$$
			
			Expanding out we get
$$
0=\sum_{j=0}^{2k+1} (\langle |\vec{a}| \vec{e}_1 , {\eta_j}\rangle)^{2k-1}+(2k-1)\left( \frac{\varepsilon}{2} \langle Q^{T} A Q({\eta_j}) , {\eta_j}\rangle -\frac{\varepsilon}{2} d \right)  \left(\langle|\vec{a}| \vec{e}_1 , {\eta_j} \rangle \right)^{2k-2} + o(\epsilon)
$$				
			
			Next we notice that if $\eta_j$ is one of the vectors then so too is $-\eta_j$ which means the leading term in the summation disappears.

		Solving for $d$ and letting $\epsilon \rightarrow 0$ we see that

		$$
d(\varepsilon) \rightarrow d= \frac{\sum_{j=0}^{2k+1}\left\langle \vec{e_{1}}, \eta_{j}\right\rangle^{2 k-2}\left\langle Q^{\top} A Q \eta_{j}, \eta_{j}\right\rangle}{\sum_{j=0}^{2k+1}\left\langle \vec{e_{1}}, \eta_{j}\right\rangle^{2 k-2}}
$$
		
	We set $U= Q^{\top} A Q$ and denote the coordinates of $\eta_j = (\eta^1_j,\eta^2_j)$ to simplify 
			
		 	$$
 d= \frac{\sum_{j=0}^{2k+1}(\eta^1_j)^{2 k-2}\left\langle U \eta_{j}, \eta_{j}\right\rangle}{\sum_{j=0}^{2k+1}(\eta^1_j)^{2 k-2}} = \frac{\sum_{j=0}^{2k+1}(\eta^1_j)^{2 k-2}(u_{11}(\eta_j^1)^2 +2u_{12}\eta_j^1 \eta_j^2 + u_{22}(\eta_j^2)^2) }{\sum_{j=0}^{2k+1}(\eta^1_j)^{2 k-2}}
$$
	Where the $u_{ij}$ are the entries in the symmetric matrix $U$. Recall $\eta_j = (cos(\theta_j),sin(\theta_j))$ where the $2k+2$ angles are uniformly distributed around the circle starting at some $\eta$. 

	Written in terms of trigonometric functions 
	
		$$
 d = \frac{\sum_{j=0}^{2k+1} \cos(\theta_j)^{2 k-2}(u_{11}\cos(\theta_j)^2 +2u_{12}\cos(\theta_j) \sin(\theta_j) + u_{22}\sin(\theta_j)^2) }{\sum_{j=0}^{2k+1}( \cos(\theta_j))^{2 k-2}}
$$
		and multiplying through 
		
		$$
 d = \frac{\sum_{j=0}^{2k+1} u_{11} \cos^{2k}(\theta_j) +2u_{12}\cos^{2k-1}(\theta_j) \sin(\theta_j) + u_{22}(1-\cos^2(\theta_j)))\cos^{2k-2}(\theta_j) }{\sum_{j=0}^{2k+1}( \cos(\theta_j))^{2 k-2}}
$$
			
	The term with $u_{12}$ is an odd function over a symmetric set of angles and becomes 0. Regrouping together the other terms

		$$
 d = \frac{\sum_{j=0}^{2k+1}  \cos^{2k-2}(\theta_j)  - \cos^{2k}(\theta_j) }{\sum_{j=0}^{2k+1}( \cos(\theta_j))^{2 k-2}}(u_{11}+u_{22}) +  \frac{\sum_{j=0}^{2k+1}  2\cos^{2k}(\theta_j)  - \cos^{2k-2}(\theta_j) }{\sum_{j=0}^{2k+1}( \cos(\theta_j))^{2 k-2}}(u_{11})
$$
			
Now $u_{11}+u_{22} = tr(U) = tr(A)$ while $u_{11} = \left\langle A \frac{\vec{a}}{|\vec{a}|}, \frac{\vec{a}}{|\vec{a}|}\right\rangle$.  
A trigonometric identity for angles in arithmetic progression (see appendix)  implies
		  \begin{equation}
		  \sum \cos^{2k}(\theta_j) =\frac{2k-1}{2k}\sum \cos^{2k-2}{\theta_j}
		 \end{equation}

which yields

			\begin{equation}\label{eq:equation2.2}
 d = (1-\frac{2k-1}{2k})tr(A) +  (2(\frac{2k-1}{2k})-1)\left\langle A \frac{\vec{a}}{|\vec{a}|}, \frac{\vec{a}}{|\vec{a}|}\right\rangle 
            \end{equation}

Recalling the formal representation p-Laplace operator from equation (2.9) and  
\begin{equation}\label{eq:}
\Delta_{p}^{G} u=\frac{1}{p} \Delta_{2} u+\frac{p-2}{p}|\nabla u|^{-2} \sum_{i, j} \frac{\partial u}{\partial x_{i}} \frac{\partial u}{\partial x_{j}} \frac{\partial^{2} u}{\partial x_{i} \partial x_{j}}
\end{equation}
shows $d = \Delta_{p}^{G} \phi $ as desired. 
	\end{proof}

  \vspace{0.2in}  
 
\subsection{Special p=4 case in three dimensions}\label{sec:p=4}

We begin to study the 3 dimensional case over the icosahedron, will show how the argument for \cref{thm:Theorem2.6} generalizes. From now on, our families of p-averages will be over a discrete set of vectors, scaled by $\varepsilon$.  

\begin{theorem}\label{thm:th1}
In three dimensions, $ A_{\varepsilon}^{4}$ has the asymptotic mean value property for the $4$-Laplacian  where the set of vertices $J$ is given by the Icosahedron. 
\end{theorem}
\begin{proof}

 $J=\{\eta_1,\eta_2,\cdots , \eta_{12}\} = \{ (0,\pm1, \pm \phi),(\pm 1,\pm \phi,0),(\pm \phi,0,\pm 1) \} $ where $\phi$ is the golden ratio, that is $\phi = \frac{1 + \sqrt{5}}{2}$, the positive solution to $x^2-x-1=0$. With this set of vertices

$$ \frac{A_{\varepsilon}^{4}[\phi](x) - \phi(x)}{\frac{\varepsilon^2}{2}}= \argmin_{d \in \mathbb{R}} \sum_{j \in J}\left|(\phi(x + \varepsilon \eta_j)-\phi(x))-\frac{\varepsilon^2}{2} d\right|^{4}.$$

By continuity we will only consider quadratic functions 
$$ \phi(y)=\phi(x)+\langle \vec{a} , (y-x)\rangle+\frac{1}{2}\langle A(y-x), y-x\rangle  $$ with  $\vec{u}=\nabla \phi (x) \neq 0$ and $A=\nabla^{2} \phi = D^2 \phi$ so that repeating previous steps 
			
				$$
			\frac{A_{\varepsilon}^{4}[\phi](x) - \phi(x)}{\frac{\varepsilon^2}{2}} = \argmin_{d \in \mathbb{R}} \sum_{j \in J}\left|(\langle \vec{u} , \eta_j\rangle + \frac{\varepsilon}{2} \langle A \eta_j , \eta_j\rangle-\frac{\varepsilon}{2} d  \right|^{4}.
			$$

\vspace{0.2in}
\noindent			
			Taking derivatives we see that $d = d(\epsilon)$ satisfies
			
			$$
0=\sum_{j \in J}\left(\langle \vec{u} , \eta_j \rangle+ \frac{\varepsilon}{2} \langle A \eta_j , \eta_j\rangle -\frac{\varepsilon}{2} d  \right)^{3}.
$$
			
\noindent
Expanding out and switching notation for inner product, we get

			$$
0=\sum_{j \in J}\left(\left<\vec{u},\eta_j\right>^{3} +3 \left(\frac{\varepsilon}{2}\left<\vec{u},\eta_j\right>^2\left<A\eta_j,\eta_j\right>-\frac{\varepsilon}{2} d \right)\left<\vec{u},\eta_j\right>^{2}\right) + o(\epsilon)
$$	
			If $\eta_j$ is a vector in $J$ then so too is $-\eta_j$ and hence the leading term in the summation disappears.  Solving for $d$ and letting $\epsilon \rightarrow 0$ we see that

		$$
d(\varepsilon) \rightarrow d=\frac{\sum_{j=1}^{12}\left<\vec{u},\eta_j\right>^2\left<A\eta_j,\eta_j\right>}{\sum_{j=1}^{j=12} \left<\vec{u},\eta_j\right>^2}
$$
		
By scaling in both the numerator and denominator we can assume  $\vec{u}=\begin{pmatrix}
u_1\\u_2\\ u_3
\end{pmatrix}$ is a unit vector and $A_{3x3}= \begin{pmatrix}
a_{11} &a_{12}&a_{13}\\a_{21} &a_{22}&a_{23}\\ a_{31} &a_{32}&a_{33}
\end{pmatrix}$ a symmetric matrix.\\

Let\\
$$\eta_1=\begin{pmatrix}
1\\\phi\\ 0
\end{pmatrix}, \eta_2=\begin{pmatrix}
0\\1\\ \phi
\end{pmatrix}, \eta_3=\begin{pmatrix}
\phi\\0\\ 1
\end{pmatrix}, \eta_4=\begin{pmatrix}
1\\-\phi\\ 0
\end{pmatrix}, \eta_5=\begin{pmatrix}
0\\1\\ -\phi
\end{pmatrix}, \eta_6=\begin{pmatrix}
-\phi\\0\\ 1
\end{pmatrix}$$\\
$$\eta_{7}=-\eta_{4},\eta_{8}=-\eta_{5},\eta_{9}=-\eta_{6},\eta_{10}=-\eta_{1},\eta_{11}=-\eta_{2},\eta_{12}=-\eta_{3}.$$

From symmetry and the relationship  $\phi^2 = \phi + 1$ one gets
\begin{flalign*}
 \displaystyle 
 \left<A\eta_1,\eta_1\right>&\\&=\left<\begin{pmatrix}
a_{11} &a_{12}&a_{13}\\a_{21} &a_{22}&a_{23}\\ a_{31} &a_{32}&a_{33}
\end{pmatrix}\begin{pmatrix}
1\\\phi\\ 0
\end{pmatrix},\begin{pmatrix}
1\\\phi\\ 0
\end{pmatrix}\right>=\left<\begin{pmatrix}
a_{11} + \phi a_{12}\\a_{21} +\phi a_{22}\\ a_{31} +\phi a_{32}
\end{pmatrix},\begin{pmatrix}
1\\\phi\\ 0
\end{pmatrix}\right>&\\&=a_{11} + 2\phi a_{12}+\phi a_{22}+a_{22}=\left<A\eta_{10},\eta_{10}\right>
 \end{flalign*}
Similar calculations show 

$\left<A\eta_2,\eta_2\right>=a_{22} + 2\phi a_{23}+\phi a_{33}+a_{33}=\left<A\eta_{11},\eta_{11}\right>$\\

$\left<A\eta_3,\eta_3\right>=a_{11} + 2\phi a_{13}+\phi a_{11}+a_{33}=\left<A\eta_{12},\eta_{12}\right>$\\

 $\left<A\eta_4,\eta_4\right>=a_{11} - 2\phi a_{12}+\phi a_{22}+a_{22}=\left<A\eta_{7},\eta_{7}\right>$\\

 $\left<A\eta_5,\eta_5\right>=a_{22} - 2\phi a_{23}+\phi a_{33}+a_{33}=\left<A\eta_{8},\eta_{8}\right>$\\

$\left<A\eta_6,\eta_6\right>=a_{11} - 2\phi a_{13}+\phi a_{11}+a_{33}=\left<A\eta_{9},\eta_{9}\right>$\\

Also
$$\left<\vec{u},\eta_1\right>^2=u_1^2+2\phi u_1u_2+u_2^2+\phi u_2^2=\left<u,\eta_{10}\right>^2$$

$$\left<\vec{u},\eta_2\right>^2=u_2^2+2\phi u_2u_3+u_3^2+\phi u_3^2=\left<\vec{u},\eta_{11}\right>^2$$

$$\left<\vec{u},\eta_3\right>^2=u_1^2+\phi u_1^2+2\phi u_1u_3+u_3^2=\left<\vec{u},\eta_{12}\right>^2$$

$$\left<\vec{u},\eta_4\right>^2=u_1^2-2\phi u_1u_2+u_2^2+\phi u_2^2=\left<\vec{u},\eta_{7}\right>^2$$

$$\left<\vec{u},\eta_5\right>^2=u_2^2-2\phi u_2u_3+u_3^2+\phi u_3^2=\left<\vec{u},\eta_{8}\right>^2$$

$$\left<\vec{u},\eta_6\right>^2=u_1^2+\phi u_1^2-2\phi u_1u_3+u_3^2=\left<\vec{u},\eta_{9}\right>^2$$

\begin{equation}
  \sum_{j=1}^{j=12} \left<\vec{u},\eta_j\right>^2=2\sum_{j=1}^{j=6} \left<\vec{u},\eta_j\right>^2=2(4+2 \phi)=8+4\phi=4+4\phi^2  
\end{equation}

\begin{flalign*}
 \displaystyle
 &\\&\left<\vec{u},\eta_1\right>^2\left<A\eta_1,\eta_1\right>+\left<\vec{u},\eta_4\right>^2\left<A\eta_4,\eta_4\right>&\\&=2u_1^2a_{11}+8\phi ^2 u_1 u_2 a_{12}+2\phi ^2 a_{22}u_1^2+2\phi^2 u_2^2a_{11}+(6\phi +4)u_2^2a_{22}
 \end{flalign*}

\begin{flalign*}
 \displaystyle
 &\\&\left<\vec{u},\eta_2\right>^2\left<A\eta_2,\eta_2\right>+\left<\vec{u},\eta_5\right>^2\left<A\eta_5,\eta_5\right>&\\&=2u_2^2a_{22}+8\phi ^2 u_2 u_3 a_{23}+2\phi ^2 a_{33}u_2^2+2\phi^2 u_3^2a_{22}+(6\phi +4)u_3^2a_{33}
 \end{flalign*}

\begin{flalign*}
 \displaystyle
 &\\&\left<\vec{u},\eta_3\right>^2\left<A\eta_3,\eta_3\right>+\left<\vec{u},\eta_6\right>^2\left<A\eta_6,\eta_6\right>&\\&=2u_3^2a_{33}+8\phi ^2 u_3 u_1 a_{13}+2\phi ^2 a_{11}u_3^2+2\phi^2 u_1^2a_{33}+(6\phi +4)u_1^2a_{11}
 \end{flalign*}

\begin{flalign*}
  \sum_{j=1}^{6}\left<\vec{u},\eta_j\right>^2\left<A\eta_j,\eta_j\right>&\\&=2\sum_{j=1}^{3}u_j^2a_{jj}+8\phi ^2\sum_{i,j=1,i\neq j}^{3} u_i u_j a_{ij}&\\&+2\phi ^2 \sum_{i,j=1,i\neq j}^{3}a_{ii}u_j^2+(6\phi +4)\sum_{j=1}^{3}u_j^2a_{jj}&\\&=4\phi ^2 \left(\sum_{j=1}^{3}u_j^2a_{jj}+2\sum_{i,j=1,i\neq j}^{3} u_i u_j a_{ij}\right)&\\&+2\phi ^2 \left(\sum_{j=1}^{3}u_j^2a_{jj}+\sum_{i,j=1,i\neq j}^{3}a_{ii}u_j^2\right)&\\&=4\phi ^2 \left<A\vec{u},\vec{u}\right>+2\phi ^2 tr(A)&\\
\end{flalign*}

Thus 
\begin{flalign*}
 d&=\frac{\sum_{j=1}^{12}\left<\vec{u},\eta_j\right>^2\left<A\eta_j,\eta_j\right>}{\sum_{j=1}^{12} \left<\vec{u},\eta_j\right>^2}&\\&=\frac{2\left(4\phi ^2 \left<A\vec{u},\vec{u}\right>+2\phi ^2 tr(A)\right)}{4(1+\phi^2)}&\\&=\frac{4\phi^2}{4(1+\phi^2)}\left(2\left<A\vec{u},\vec{u}\right>+tr(A)\right)&\\&=(\phi^2+1)\frac{4 \phi^2}{(\phi^2+1)^2}\left(\frac{1}{4} \Delta \vec{u} +\frac{4-2}{4}\Delta_\infty\right)&\\&   
 =(\phi^2+1)\frac{4 (\phi +1)}{(\phi+2)^2}\left(  \Delta_4^G \vec{u}  \right)&\\&
 =(\phi^2+1)\frac{4 (\phi +1)}{(5 \phi+5)}\left(  \Delta_4^G \vec{u}  \right)
 &
\end{flalign*}
Similar to the result we have in equation\cref{eq:equation2.2}.  
\end{proof}
If we had used normalized unit vectors in $J$, that is dividing by $\sqrt{1+\phi^2}$ we would have a set with $c_{4,3} = \frac{1}{2}d = \frac{2}{5}$ in agreement with the AMVP for the boundary of the sphere in equation (3.5).

\vspace{0.2in}
We already examined the averaging behavior for an icosahedron in 3 dimensions.  Of course a cube is another polyhedron with averaging behavior.  The set of regular polytopes was investigated for averaging behaviors for harmonic polynomials in \cite{Japanese} which inspired our investigations into higher dimensional, discrete p-averages for p an even integer. Upon investigation of all the previous arguments one sees the common property is a matrix equality which we now define.

\begin{definition} 
		We say that a discrete set of vectors $J = \{\eta_1, \eta_2, \cdots \eta_k \} \subset \mathbb{R}^n$ is a \textbf{ p-averaging set } if there is a $d_{p,n}$ so that for every symmetric matrix $A$, and unit vector $\vec{u}$  we have
		\begin{equation}
		d_{p,n}(\frac{1}{p} tr(A) + \frac{p-2}{p}\left<A \vec{u},\vec{u}\right>)=\frac{\sum_{j \in J} \left<\vec{u},\eta_j\right>^{p-2}\left<A\eta_j,\eta_j\right>}{\sum_{j \in J}  \left<\vec{u},\eta_j\right>^{p-2}} 
		\end{equation}
	\end{definition}

The next example requires extensive computations which are easily done using MATLAB.  This example corresponds to the 20 vertices on the dodecahedron.

\begin{example}\label{example2}
  $J=\{( 0, \pm 1/c, \pm c), (\pm 1/c, \pm c, 0), (\pm c, 0, \pm 1/c),(\pm 1, \pm 1 \pm 1)\}$, the set of vertices of the dodecahedron where $c = \frac{1 + \sqrt{5}}{2}$ is a $4$- \textbf{averaging set}.

\end{example}

It can be shown that for any unit vector $\vec{u}$, one has 
$\sum_{j=1}^{20} \left<\vec{u},\eta_j\right>^2=20(u_1^2+u_2^2+u_3^2)$  \hspace{0.1in} while
$\sum_{j=1}^{20}\left<\vec{u},\eta_j\right>^2\left<A\eta_j,\eta_j\right>=24\left<A \vec{u},\vec{u}\right>+12\text{ tr}(A)(u_1^2+u_2^2+u_3^2)$

Thus

\begin{align*}
&\frac{\sum_{j=1}^{20}\left<\vec{u},\eta_j\right>^2\left<A\eta_j,\eta_j\right>}{\sum_{j=1}^{20} \left<\vec{u},\eta_j\right>^2}&\\\\&=\frac{\left(24\left<A\vec{u},\vec{u}\right>+12\text{ tr}(A)\right)}{20}&\\\\&= \frac{12}{20} \left(2\left<A\vec{u},\vec{u}\right>+\text{ tr}(A)\right)&\\\\&= 3*\frac{4}{5} \left(\frac{2}{4}\left<A\vec{u},\vec{u}\right>+\frac{1}{4}\text{ tr}(A)\right)&\\    
\end{align*}

so that $J$ is a $4-$ \textbf{averaging} set in 3 dimensions. 

\subsection{Four Dimensions}\label{sec:4dim}

In this section we show the AMVP for a set of polytopes in 4 dimensions.  The most interesting of these perhaps is called the 24-cell. One of the most interesting facts about the 24-cell is that it is the basis for a tessellating region in four dimensions, called the honeycomb which would allow numerical examples to be computed.  To date, most numerics have been restricted to two dimensions as the angular variables typically are never fully discretized. The full description of the 24-cell honeycomb is a generalization of the hexagonal tiling of two dimensions and is best described algebraically, a partial geometric rendering is given in figure 1. 

We begin with the description

\begin{theorem}\label{thm:th2}
In four dimensions, $ A_{\varepsilon}^{4}$ has the asymptotic mean value property for the $4$-Laplacian  where the set of vertices $J$ is given by the 24-cell honeycomb, or icositetrachoric honeycomb. 
\end{theorem}
\begin{proof}

In four-dimensional Euclidean geometry, the 24-cell honeycomb, or icositetrachoric honeycomb is a regular space-filling tessellation of 4-dimensional space by regular 24-cells. It is the only tessellating domain besides the obvious 4-cubic honeycomb.  It has the AMVP for p=4.    

The vertices on a 24-cell centered at $(0,0,0,0)$ can be chosen as $$(1,1,0,0),(1,0,1,0),(1,0,0,1),(0,1,1,0),(0,1,0,1),(0,0,1,1)$$ and all changes where the 1s are replaced by $\pm 1$.

\begin{figure}[ht]
    \centering
    \includegraphics[width=0.75\textwidth]{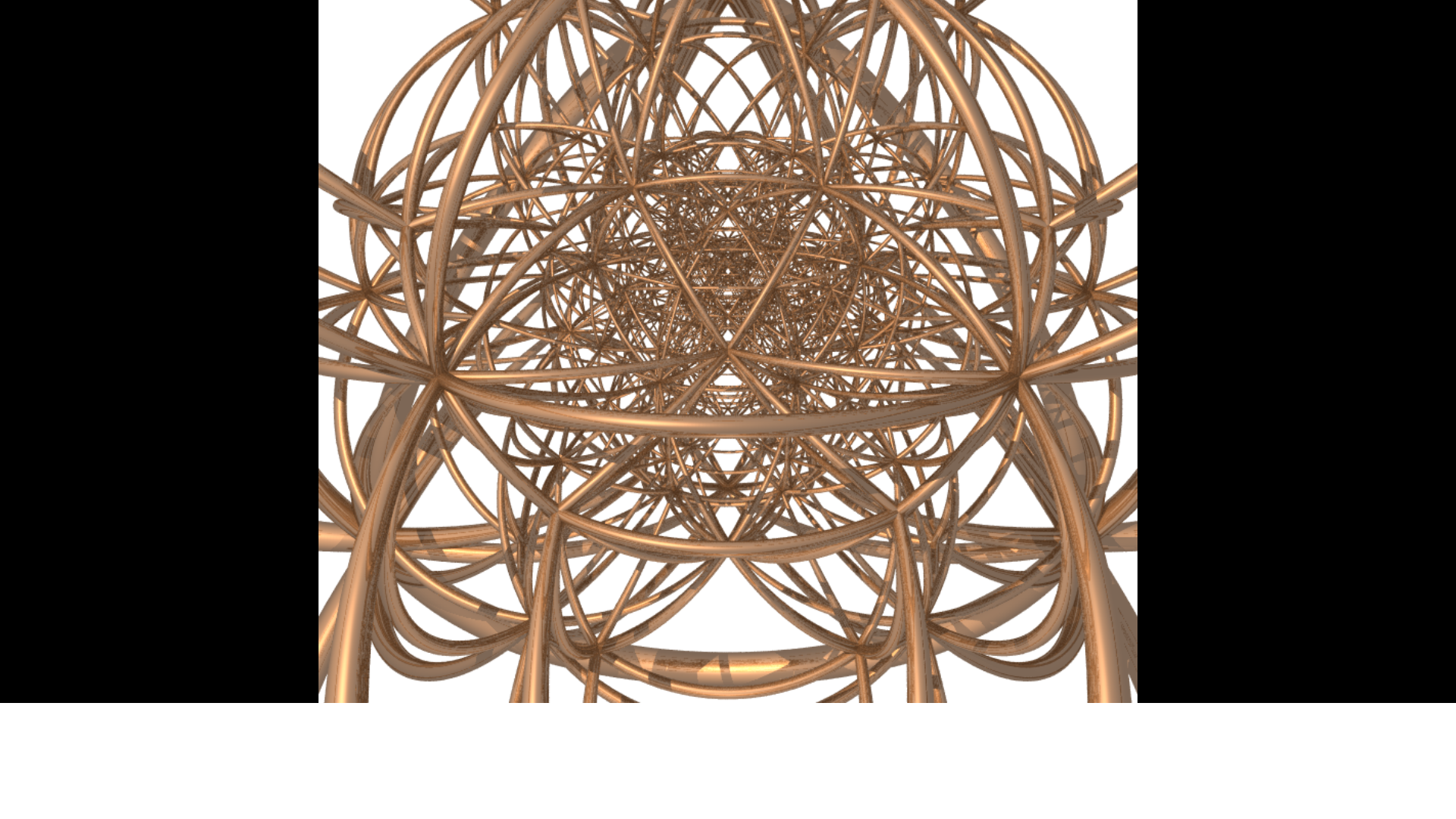}
    \caption{Honeycomb 24-cell and layer of adjacent faces }
    \label{fig:fig6}
\end{figure}

\vspace{1cm}

Let $\vec{u}=\begin{pmatrix}
u_1\\u_2\\ u_3\\ u_4
\end{pmatrix}$ be a unit vector and $A_{4x4}= \begin{pmatrix}
a_{11} &a_{12}&a_{13}&a_{14}\\a_{21} &a_{22}&a_{23}&a_{24}\\ a_{31} &a_{32}&a_{33}&a_{34}\\a_{41} &a_{42}&a_{43}&a_{44}
\end{pmatrix}$ be a symmetric matrix and set
\begin{flalign*}
\eta_1=\begin{pmatrix}
1\\1\\0\\0
\end{pmatrix}
\hspace{0.13in}
\eta_2=\begin{pmatrix}
1\\0\\1\\0
\end{pmatrix}
\hspace{0.13in}
\eta_3=\begin{pmatrix}
1\\0\\0\\1
\end{pmatrix}
\hspace{0.13in}
\eta_4=\begin{pmatrix}
0\\1\\1\\0
\end{pmatrix}
\hspace{0.13in}
\eta_5=\begin{pmatrix}
0\\1\\0\\1
\end{pmatrix}
\hspace{0.13in}
\eta_6=\begin{pmatrix}
0\\0\\1\\1
\end{pmatrix}
\end{flalign*}

\begin{flalign*}
\eta_7=\begin{pmatrix}
1\\-1\\0\\0
\end{pmatrix}
 \eta_8=\begin{pmatrix}
1\\0\\-1\\0
\end{pmatrix} \eta_9=\begin{pmatrix}
1\\0\\0\\-1
\end{pmatrix} \eta_{10}=\begin{pmatrix}
0\\1\\-1\\0
\end{pmatrix} \eta_{11}=\begin{pmatrix}
0\\1\\0\\-1
\end{pmatrix} \eta_{12}=\begin{pmatrix}
0\\0\\1\\-1
\end{pmatrix}
\end{flalign*}

\begin{flalign*}
 \eta_{13}=-\eta_{7},
 \eta_{14}=-\eta_{8}, 
 \eta_{15}=-\eta_{3}, 
 \eta_{16}=-\eta_{10}, \eta_{17}=-\eta_{11}, \eta_{18}=-\eta_{12}, 
\end{flalign*}

$$
\eta_{19}=-\eta_{1}
,\eta_{20}=-\eta_{2},\eta_{21}=-\eta_{3},\eta_{22}=-\eta_{4},\eta_{23}=-\eta_{5},\eta_{24}=-\eta_{6} 
$$

Because of symmetry we have
\begin{flalign*}
\left<A\eta_1,\eta_1\right>&\\&=\left<\begin{pmatrix}
a_{11} &a_{12}&a_{13}&a_{14}\\a_{21} &a_{22}&a_{23}&a_{24}\\ a_{31} &a_{32}&a_{33}&a_{34}\\a_{41} &a_{42}&a_{43}&a_{44}
\end{pmatrix}\begin{pmatrix}
1\\1\\ 0\\0
\end{pmatrix},\begin{pmatrix}
1\\1\\ 0\\0
\end{pmatrix}\right>=\left<\begin{pmatrix}
a_{11} + a_{12}\\a_{21} +a_{22}\\ a_{31} +a_{32}\\a_{41} +a_{42}
\end{pmatrix},\begin{pmatrix}
1\\1\\ 0\\0
\end{pmatrix}\right>&\\&=a_{11} + 2 a_{12}+a_{22}=\left<A\eta_{19},\eta_{19}\right>
\end{flalign*}

Similarly

$\left<A\eta_2,\eta_2\right>=a_{11} + 2 a_{13}+a_{33}=\left<A\eta_{20},\eta_{20}\right>$

$\left<A\eta_3,\eta_3\right>=a_{11} + 2 a_{14}+a_{44}=\left<A\eta_{21},\eta_{21}\right>$

$\left<A\eta_4,\eta_4\right>=a_{22} + 2 a_{23}+a_{33}=\left<A\eta_{22},\eta_{22}\right>$

$\left<A\eta_5,\eta_5\right>=a_{22} + 2 a_{24}+a_{44}=\left<A\eta_{23},\eta_{23}\right>$

$\left<A\eta_6,\eta_6\right>=a_{33} + 2 a_{34}+a_{44}=\left<A\eta_{24},\eta_{24}\right>$

$\left<A\eta_7,\eta_7\right>=a_{11} - 2 a_{12}+a_{22}=\left<A\eta_{13},\eta_{13}\right>$

$\left<A\eta_8,\eta_8\right>=a_{11} - 2 a_{13}+a_{33}=\left<A\eta_{14},\eta_{14}\right>$

$\left<A\eta_9,\eta_9\right>=a_{11} - 2 a_{14}+a_{44}=\left<A\eta_{15},\eta_{15}\right>$

$\left<A\eta_{10},\eta_{10}\right>=a_{22} - 2 a_{23}+a_{33}=\left<A\eta_{16},\eta_{16}\right>$

$\left<A\eta_{11},\eta_{11}\right>=a_{22} - 2 a_{24}+a_{44}=\left<A\eta_{17},\eta_{17}\right>$

$\left<A\eta_{12},\eta_{12}\right>=a_{33} - 2 a_{34}+a_{44}=\left<A\eta_{18},\eta_{18}\right>$

\vspace{0.5cm}

Now we have,

$$\left<\vec{u},\eta_1\right>^2=\left<\begin{pmatrix}
u_1\\u_2\\ u_3\\u_4
\end{pmatrix},\begin{pmatrix}
1\\1\\0\\0
\end{pmatrix}\right>^2=(u_1+u_2)^2=u_1^2+2 u_1u_2+u_2^2=\left<\vec{u},\eta_{14}\right>^2$$

Similarly we have:\\

$\left<\vec{u},\eta_2\right>^2=u_1^2+2 u_1u_3+u_3^2=\left<\vec{u},\eta_{20}\right>^2$\\

$\left<\vec{u},\eta_3\right>^2=u_1^2+2 u_1u_4+u_4^2=\left<\vec{u},\eta_{21}\right>^2$\\

$\left<\vec{u},\eta_4\right>^2=u_2^2+2 u_2u_3+u_3^2=\left<u,\eta_{22}\right>^2$\\

$\left<\vec{u},\eta_5\right>^2=u_2^2+2 u_2u_4+u_4^2=\left<\vec{u},\eta_{23}\right>^2$\\

$\left<\vec{u},\eta_6\right>^2=u_3^2+2 u_3u_4+u_4^2=\left<\vec{u},\eta_{24}\right>^2$\\

$\left<\vec{u},\eta_7\right>^2=u_1^2-2 u_1u_2+u_2^2=\left<\vec{u},\eta_{13}\right>^2$\\

$\left<\vec{u},\eta_8\right>^2=u_1^2-2 u_1u_3+u_3^2=\left<\vec{u},\eta_{14}\right>^2$\\

$\left<\vec{u},\eta_9\right>^2=u_1^2-2 u_1u_4+u_4^2=\left<\vec{u},\eta_{15}\right>^2$\\

$\left<\vec{u},\eta_{10}\right>^2=u_2^2-2 u_2u_3+u_3^2=\left<\vec{u},\eta_{16}\right>^2$\\

$\left<\vec{u},\eta_{11}\right>^2=u_2^2-2 u_2u_4+u_4^2=\left<\vec{u},\eta_{17}\right>^2$\\

$\left<\vec{u},\eta_{12}\right>^2=u_3^2-2 u_3u_4+u_4^2=\left<\vec{u},\eta_{18}\right>^2$\\

\begin{equation}
  \sum_{j=1}^{j=12} \left<\vec{u},\eta_j\right>^2=6  
\end{equation}

$$\left<\vec{u},\eta_1\right>^2\left<A\eta_1,\eta_1\right>+\left<\vec{u},\eta_7\right>^2\left<A\eta_7,\eta_7\right>=2u_1^2(a_{11}+a_{22})+2u_2^2(a_{11}+a_{22})+8u_1u_2a_{12}$$

$$\left<\vec{u},\eta_2\right>^2\left<A\eta_2,\eta_2\right>+\left<\vec{u},\eta_8\right>^2\left<A\eta_8,\eta_8\right>=2u_1^2(a_{11}+a_{33})+2u_3^2(a_{11}+a_{33})+8u_1u_3a_{13}$$

$$\left<\vec{u},\eta_3\right>^2\left<A\eta_3,\eta_3\right>+\left<\vec{u},\eta_9\right>^2\left<A\eta_9,\eta_9\right>=2u_1^2(a_{11}+a_{44})+2u_4^2(a_{11}+a_{44})+8u_1u_4a_{14}$$

$$\left<\vec{u},\eta_4\right>^2\left<A\eta_4,\eta_4\right>+\left<\vec{u},\eta_{10}\right>^2\left<A\eta_{10},\eta_{10}\right>=2u_2^2(a_{22}+a_{33})+2u_3^2(a_{22}+a_{33})+8u_2u_3a_{23}$$

$$\left<\vec{u},\eta_5\right>^2\left<A\eta_5,\eta_5\right>+\left<\vec{u},\eta_{11}\right>^2\left<A\eta_{11},\eta_{11}\right>=2u_2^2(a_{22}+a_{44})+2u_4^2(a_{22}+a_{44})+8u_2u_4a_{24}$$

$$\left<\vec{u},\eta_6\right>^2\left<A\eta_6,\eta_6\right>+\left<\vec{u},\eta_{12}\right>^2\left<A\eta_{12},\eta_{12}\right>=2u_3^2(a_{33}+a_{44})+2u_4^2(a_{33}+a_{44})+8u_3u_4a_{34}$$
Now we have
\begin{flalign*}
    \sum_{j=1}^{12}\left<\vec{u},\eta_j\right>^2\left<A\eta_j,\eta_j\right>&=4\left(\sum_{j=1}^4 u_j^2 a_{jj}+2\sum_{i,j=1,i\neq j}^4 u_i u_j a_{ij}\right)&\\&+2u_1^2(a_{11}+a_{22}+a_{33}+a_{44})&\\&+2u_2^2(a_{11}+a_{22}+a_{33}+a_{44})+2u_3^2(a_{11}+a_{22}+a_{33}+a_{44})&\\&+2u_4^2(a_{11}+a_{22}+a_{33}+a_{22})&\\&=4\left<A\vec{u},\vec{u}\right>+2(u_1^2+u_2^2+u_3^2+u_4^2)(a_{11}+a_{22}+a_{33}+a_{44})&\\&=4\left<A\vec{u},\vec{u}\right>+2\text{ tr}(A)&\\
\end{flalign*}

Thus
\begin{align*}
    d&=\frac{\sum_{j=1}^{12}\left<\vec{u},\eta_j\right>^2\left<A\eta_j,\eta_j\right>}{\sum_{j=1}^{12} \left<\vec{u},\eta_j\right>^2}&\\\\&=\frac{\left(4\left<A\vec{u},\vec{u}\right>+2\text{ tr}(A)\right)}{6}&\\\\&=\frac{8}{6}\left(\frac{2}{4}\left<A\vec{u},\vec{u}\right>+\frac{1}{4}\text{ tr}(A)\right)&\\
\end{align*}
\end{proof}

We remark that by factoring out a 2, because of the square of the norm of the vectors $\eta$ and then multiplying by $\frac{1}{2}$, one would arrive at $c_{4,4} = \frac{2}{6}$, in agreement with equation (3.5).

\vspace{0.2in} 

There are several ways to construct the 24-cell honeycomb. The most algebraic is to take the Hurwitz quaternions with even square norm as center.  The interiors of the 24-cells are the locus of points closest to one of these centers, that is one constructs a Voronoi tesellation. Extremal points of these polyhedron are the vertices of the 24-cells which are located at the Hurwitz quaternions with odd square norm.  

\vspace{0.1in}
\noindent
For example $(0,0,0,0)$ would be a center of a 24-cell with 8 of its vertices given by the permutations of $(\pm 1,0,0,0)$ and the other 16 vertices given by $(\pm \frac{1}{2}, \pm \frac{1}{2}, \pm \frac{1}{2}, \pm \frac{1}{2}). $  Some centers which have distance squared 2 away, such as $(1,1,0,0)$ would share 6 vertices while other centers such as $(2,0,0,0)$ would share only 1 vertex, in this case $(1,0,0,0)$.

\vspace{0.2in}
In the final two examples, we examined the other regular polytopes in 4 dimensions which were also considered in \cite{Japanese}.  These figures are beyond the ability to keep track of by hand and MATLAB code was used to examine asymptotic mean value properties.   In both cases, the AMVP was found to hold for p=4.  The first polytope under consideration is the $600$-cell. It is a convex regular polytope in 4 dimensions which is analogous to a Platonic solid. Its boundary is composed of 600 tetrahedral cells with 20 meeting at each vertex. There are 120 vertices which we now describe.

\begin{example}\label{example3}

The vertex set for the $600$-cell is  $(\pm 1, \pm 1, \pm 1, \pm 1),  (\pm 2,  0, 0, 0)$ and its permutations, and the even permutation for $(\pm c,\pm 1, \pm 1/c, 0)$ where
$c = \frac{1 + \sqrt{5}}{2}$ is a $4$-\textbf{averaging} set.
\end{example}

We have
$\sum_{j=1}^{120} \left<\vec{u},\eta_j\right>^2=120(u_1^2+u_2^2+u_3^2+u_4^2)$ and\\

$\sum_{j=1}^{120}\left<\vec{u},\eta_j\right>^2\left<A\eta_j,\eta_j\right>=160\left<A\vec{u},\vec{u}\right>+80\text{ tr}(A)(u_1^2+u_2^2+u_3^2+u_4^2)$

\vspace{0.1in}
Thus $J$ is a $4-$\textbf{averaging domain}
\begin{align*}
& \frac{\sum_{j=1}^{120}\left<\vec{u},\eta_j\right>^2\left<A\eta_j,\eta_j\right>}{\sum_{j=1}^{120} \left<\vec{u},\eta_j\right>^2}&\\\\&=\frac{\left(160\left<A\vec{u},\vec{u}\right>+80\text{ tr}(A)\right)}{120}&\\\\&=  \frac{80}{120} \left(2\left<A\vec{u},\vec{u}\right>+\text{ tr}(A)\right)&\\\\&=  \frac{8}{3} \left(\frac{2}{4}\left<A\vec{u},\vec{u}\right>+ \frac{1}{4}\text{ tr}(A)\right)&\\
\end{align*}

The final polytope we look at is the $120$-cell. It is a regular convex polytope whose boundary is composed of 120 dodecahedral cells with 4 meeting at each vertex. It has 600 vertices in its vertex set $J$. 

\begin{example}\label{example4}
The vertex set for the $120$-cell is the set of all permutations for\\ $(\pm 2, \pm 2, 0, 0), (\pm \sqrt{5},  \pm 1, \pm 1, \pm 1)$,  $(\pm c,\pm c, \pm c, \pm c^{-2}), \text{ and } (\pm c^2,\pm c^{-1}, \pm c^{-1}, \pm c^{-1})$,\\ and the even permutations for $(\pm c^2,\pm c^{-2}, \pm 1, 0), (\pm \sqrt{5},\pm c^{-1}, \pm c, \pm 0 ),\\ (\pm 2,\pm 1, \pm c, \pm c^{-1} )$ where $c = \frac{1 + \sqrt{5}}{2}$.  It is a $4$-\textbf{averaging set}.
\end{example}

\vspace{0.1in}
\noindent
Computations show that 
$\sum_{j=1}^{600} \left<\vec{u},\eta_j\right>^2=1200(u_1^2+u_2^2+u_3^2+u_4^2)$ and\\

$\sum_{j=1}^{600}\left<\vec{u},\eta_j\right>^2\left<A\eta_j,\eta_j\right>=3200\left<A \vec{u},\vec{u}\right>+1600\text{ tr}(A)(u_1^2+u_2^2+u_3^2+u_4^2)$

\vspace{0.1in}
And the usual ratio shows that $J$ is a $4-$ \textbf{averaging} set. 
\begin{align*}
    & \frac{\sum_{j=1}^{600}\left<\vec{u},\eta_j\right>^2\left<A\eta_j,\eta_j\right>}{\sum_{j=1}^{600} \left<\vec{u},\eta_j\right>^2}&\\\\&=\frac{\left(3200\left<A \vec{u},\vec{u}\right>+1600\text{ tr}(A)\right)}{1200}&\\\\&= \frac{1600}{1200} \left(2\left<A \vec{u},\vec{u}\right>+\text{ tr}(A)\right)&\\\\&= \frac{16}{3} \left(\frac{2}{4}\left<A \vec{u},\vec{u}\right>+ \frac{1}{4}\text{ tr}(A)\right).&
\end{align*}

\vspace{0.5in}

\section{Discussion}

We have shown the first set of fully discrete vectors leading to an asymptotic mean value property for the p-Laplacian, $p \neq 2$. There are a large number of questions which this result raises.  The first is to generalize to weighted version. If one defines weighted p-averages on a discrete set of vertices, via definition 2.2 then the AMVP for this average leads one to the following definition

\begin{definition} 
		We say that a discrete set of vectors $J = \{\eta_1, \eta_2, \cdots \eta_k \} \subset \mathbb{R}^n$ is a \textbf{ weighted p-averaging set } if there is a $d_{p,n}$ and a set of weights $\{ \omega_1, \omega_2, \cdots \omega_k \}$ so that for every symmetric matrix $A$, and unit vector $\vec{u}$  we have
		\begin{equation}
		d_{p,n}(\frac{1}{p} tr(A) + \frac{p-2}{p}\left<A \vec{u},\vec{u}\right>)=\frac{\sum_{j \in J} \left<\vec{u},\eta_j\right>^{p-2}\left<A\eta_j,\eta_j\right> \omega_j}{\sum_{j \in J}  \left<\vec{u},\eta_j\right>^{p-2} \omega_j} 
		\end{equation}
	\end{definition}

\vspace{0.2in}
\begin{example}\label{example5}
In 2 dimensions, the vertex set  $(\pm \frac{1}{2}, \pm \frac{1}{2},), (\pm 1,  0)$,  $(0,\pm 1), $ with weights of $4$ for the $(\pm \frac{1}{2}, \pm \frac{1}{2},)$ vertices and $1$ otherwise is a \textbf{weighted 4-averaging set}.
\end{example}

\begin{align*}
  \frac{\sum_{j \in J} \left<\vec{u},\eta_j\right>^{p-2}\left<A\eta_j,\eta_j\right> \omega_j}{\sum_{j \in J}  \left<\vec{u},\eta_j\right>^{p-2} \omega_j}&\\&= \frac{a_{11}+a_{22}+2u_{1}^2a_{11}+2u_{2}^2a_{22}+4u_{1}u_{2}a_{12}}{6}&\\&=\frac{tr(A) + 2\left<A \vec{u},\vec{u}\right>)}{6}&\\&=\frac{4}{6}(\frac{1}{4}tr(A) + \frac{2}{4}\left<A \vec{u},\vec{u}\right>))  
\end{align*}	

\vspace{0.2in}
\begin{example}\label{example6}
In 3 dimensions, the vertex set\\  $(\pm \frac{1}{2}, \pm \frac{1}{2},\pm \frac{1}{2}), (\pm 1,  0, 0)$,  $(0,\pm 1, 0), (0, 0, \pm 1)$ with weights of $16/2^n=16/2^3$ for the $(\pm \frac{1}{2}, \pm \frac{1}{2},\pm \frac{1}{2})$ vertices and $1$ otherwise is a \textbf{weighted 4-averaging set}.
\end{example}

  $$\frac{\sum_{j \in J} \left<\vec{u},\eta_j\right>^{p-2}\left<A\eta_j,\eta_j\right> \omega_j}{\sum_{j \in J}  \left<\vec{u},\eta_j\right>^{p-2} \omega_j}$$
  $$=\frac{a_{11}+a_{22}+a_{33}+2u_{1}^2a_{11}+2u_{2}^2a_{22}+2u_{3}^2a_{33}+4u_{1}u_{2}a_{12}+4u_{1}u_{3}a_{13}+4u_{2}u_{3}a_{23}}{6}$$

  $$=\frac{tr(A) + 2\left<A \vec{u},\vec{u}\right>)}{6}=\frac{4}{6}(\frac{1}{4}tr(A) + \frac{2}{4}\left<A \vec{u},\vec{u}\right>))$$

\vspace{0.2in}
\begin{example}\label{example7}
In 4 dimensions, the vertex set\\  $(\pm \frac{1}{2}, \pm \frac{1}{2},\pm \frac{1}{2}, \pm \frac{1}{2}), (\pm 1,  0, 0, 0)$,  $(0,\pm 1, 0, 0), (0, 0, \pm 1, 0), (0, 0, 0, \pm 1)$ with weights of $16/2^n=16/2^4$ for the $(\pm \frac{1}{2}, \pm \frac{1}{2},\pm \frac{1}{2}, \pm \frac{1}{2})$ vertices and $1$ otherwise is a \textbf{weighted 4-averaging set}.
\end{example}

  $$\frac{\sum_{j \in J} \left<\vec{u},\eta_j\right>^{p-2}\left<A\eta_j,\eta_j\right> \omega_j}{\sum_{j \in J}  \left<\vec{u},\eta_j\right>^{p-2} \omega_j}$$
  $$=\frac{a_{11}+a_{22}+a_{33}+a_{44}}{6}$$
  $$+\frac{ 2(u_{1}^2a_{11}+u_{2}^2a_{22}+u_{3}^2a_{33}+u_{4}^2a_{44})}{6}$$
   $$+\frac{4(u_{1}u_{2}a_{12}+u_{1}u_{3}a_{13}+u_{1}u_{4}a_{14}+u_{2}u_{3}a_{23}+u_{2}u_{4}a_{24}+u_{3}u_{4}a_{34})}{6}$$

  $$=\frac{tr(A) + 2\left<A \vec{u},\vec{u}\right>)}{6}=\frac{4}{6}(\frac{1}{4}tr(A) + \frac{2}{4}\left<A \vec{u},\vec{u}\right>))$$

\vspace{0.2in}
\begin{example}\label{example9}
In 5 dimensions, the vertex set\\  $(\pm \frac{1}{2}, \pm \frac{1}{2},\pm \frac{1}{2}, \pm \frac{1}{2}, \pm \frac{1}{2}), (\pm 1,  0, 0, 0, 0)$, $(0,\pm 1, 0, 0, 0),(0, 0, \pm 1, 0, 0),$\\$ (0, 0, 0, \pm 1,0), (0, 0, 0, 0, \pm 1)$\\ with weights of $16/2^n=16/2^5$ for the vertices $(\pm \frac{1}{2}, \pm \frac{1}{2},\pm \frac{1}{2}, \pm \frac{1}{2},\pm \frac{1}{2})$ vertices and $1$ otherwise is a \textbf{weighted 4-averaging set}.
\end{example}

  $$\frac{\sum_{j \in J} \left<\vec{u},\eta_j\right>^{p-2}\left<A\eta_j,\eta_j\right> \omega_j}{\sum_{j \in J}  \left<\vec{u},\eta_j\right>^{p-2} \omega_j}$$
  $$=\frac{a_{11}+a_{22}+a_{33}+a_{44}+a_{55}}{6}$$
  $$+\frac{ 2(u_{1}^2a_{11}+u_{2}^2a_{22}+u_{3}^2a_{33}+u_{4}^2a_{44}+u_{5}^2a_{55})}{6}$$
   $$+\frac{4(u_{1}u_{2}a_{12}+u_{1}u_{3}a_{13}+u_{1}u_{4}a_{14}+u_{1}u_{5}a_{15})}{6}$$

  $$+\frac{4(u_{2}u_{3}a_{23}+u_{2}u_{4}a_{24}+u_{2}u_{5}a_{25})}{6}$$

$$+\frac{4(u_{3}u_{4}a_{34}+u_{3}u_{5}a_{35}+u_{4}u_{5}a_{45})}{6}$$

  $$=\frac{tr(A) + 2\left<A \vec{u},\vec{u}\right>)}{6}=\frac{4}{6}(\frac{1}{4}tr(A) + \frac{2}{4}\left<A \vec{u},\vec{u}\right>))$$

\vspace{0.2in}
\begin{example}\label{example10}
In 6 dimensions, the vertex set\\  $(\pm \frac{1}{2}, \pm \frac{1}{2},\pm \frac{1}{2}, \pm \frac{1}{2}, \pm \frac{1}{2}, \pm \frac{1}{2}), (\pm 1,  0, 0, 0, 0, 0)$, $(0,\pm 1, 0, 0, 0, 0),(0, 0, \pm 1, 0, 0, 0),$\\$ (0, 0, 0, \pm 1,0, 0), (0, 0, 0, 0, \pm 1, 0), (0, 0, 0, 0, 0,\pm 1)$\\ with weights of $16/2^n=16/2^6$ for the vertices $(\pm \frac{1}{2}, \pm \frac{1}{2},\pm \frac{1}{2}, \pm \frac{1}{2},\pm \frac{1}{2}, \pm \frac{1}{2})$ vertices and $1$ otherwise is a \textbf{weighted 4-averaging set}.
\end{example}

$$\frac{\sum_{j \in J} \left<\vec{u},\eta_j\right>^{p-2}\left<A\eta_j,\eta_j\right> \omega_j}{\sum_{j \in J}  \left<\vec{u},\eta_j\right>^{p-2} \omega_j}$$
  $$=\frac{a_{11}+a_{22}+a_{33}+a_{44}+a_{55}+a_{66}}{6}$$
  $$+\frac{ 2(u_{1}^2a_{11}+u_{2}^2a_{22}+u_{3}^2a_{33}+u_{4}^2a_{44}+u_{5}^2a_{55}+u_{6}^2a_{66})}{6}$$
   $$+\frac{4(u_{1}u_{2}a_{12}+u_{1}u_{3}a_{13}+u_{1}u_{4}a_{14}+u_{1}u_{5}a_{15}+u_{1}u_{6}a_{16})}{6}$$

  $$+\frac{4(u_{2}u_{3}a_{23}+u_{2}u_{4}a_{24}+u_{2}u_{5}a_{25}+u_{2}u_{6}a_{26})}{6}$$

$$+\frac{4(u_{3}u_{4}a_{34}+u_{3}u_{5}a_{35}+u_{3}u_{6}a_{36}+u_{4}u_{5}a_{45}+u_{4}u_{6}a_{46}+u_{5}u_{6}a_{56})}{6}$$

  $$=\frac{tr(A) + 2\left<A \vec{u},\vec{u}\right>)}{6}=\frac{4}{6}(\frac{1}{4}tr(A) + \frac{2}{4}\left<A \vec{u},\vec{u}\right>))$$

\vspace{4in}
\begin{example}\label{example8}
In 2 dimensions, the vertex set  $(\pm \frac{1}{\sqrt{2}}, \pm \frac{1}{\sqrt{2}},), (\pm 1,  0)$,  $(0,\pm 1)$. Let $p=6$
\end{example}

\begin{align*}
  \frac{\sum_{j \in J} \left<\vec{u},\eta_j\right>^{p-2}\left<A\eta_j,\eta_j\right> \omega_j}{\sum_{j \in J}  \left<\vec{u},\eta_j\right>^{p-2} \omega_j}&\\&= \frac{2.5a_{11}u_{1}^4 + 0.5a_{11}u_{2}^4+0.5a_{22}u_{1}^4 + 2.5a_{22}u_{2}^4+4a_{12}u_{1}u_{2}^3 + 4a_{12}u_{1}^3u_{2}+3a_{11}u_{1}^2u_{2}^2+ 3a_{22}u_{1}^2u_{2}^2}{3}&\\&=\frac{0.5tr(A) + 2\left<A \vec{u},\vec{u}\right>)}{3}&\\&=\frac{tr(A) + 4\left<A \vec{u},\vec{u}\right>)}{6}&\\&
\end{align*}	
\vspace{0.2in}
 
Another question for future investigation is if there is a fully discrete set for parabolic equations similar to the continuous versions in \cite{Ishiwata}.    We recall for parabolic problems one needs the definition of the heat ball

$$
E(x, t ; \varepsilon)=\left\{(y, s) \in \mathbf{R}^{n+1} \mid s < t, \frac{1}{(4 \pi(t-s))^{n / 2}} \exp \left(-\frac{|x-y|^2}{4(t-s)}\right) \geq \frac{1}{{\varepsilon}^n}\right\}
$$
  Equipping the heat ball with the space-time measure

$$
d v(y, s)=|x-y|^2d y \frac{1}{(t-s)^2} d s
$$
and following definition 2.2 one arrives at the notion of p-average of u over the heat ball, written $A_\varepsilon^p (u)(x, t)$.
In \cite{Ishiwata} the authors expanded a parabolic function into its Taylor  showed the equivalence under general conditions over a parabolic region $\Omega_T$ of

$$
u_t=\frac{np}{n+p-2} \Delta_p^G u \text { in  } \Omega_T \text { in the viscosity sense and }
$$
$$
u(x, t)=A_\varepsilon^p (u)(x, t)+o\left(\varepsilon^2\right) \text { as } \varepsilon \rightarrow 0 \text { in the viscosity sense. }
$$

The need to use the heat ball for averaging operators may not be needed as \cite{finalparabolicgames} shows for versions of the tug-of-war averaging operators.   In that article the authors show that for some averaging operators the heat ball can be replaced by $\varepsilon$-cylinders. Further questions involve whether discrete averaging operators exist in the Heisenberg group as well as how one can use these results to implement numerical schemes to solve applied problems in image analysis as was done in by Does in \cite{kerstindoes} as well as Falcone et al. \cite{SmitsAveraging}.

\newpage
\bibliographystyle{siamplain}
\bibliography{references}

\newpage
\section{Appendix- Trigonometric Identities}

\vspace{0.2in}

From $\cos \theta=\frac{e^{-i \theta}+e^{i \theta}}{2}$, together with the binomial theorem we see that for even numbers 

$$
\cos ^{2r} \theta=\frac{1}{2^{2r}}\left(\begin{array}{l}
2r \\
r
\end{array}\right)+\frac{2}{2^{2r}} \sum_{l=0}^{r-1}\left(\begin{array}{l}
2r \\
l
\end{array}\right) \cos ((2r-2 l) \theta).
$$

Also, for the sum of arguments of cosine in arithmetic progression it is easy to show

\begin{equation}
\sum_{k=0}^{n-1} \cos (\alpha+k d)=\frac{\sin (n d / 2)}{\sin (d / 2)} \cos \left(\alpha+\frac{(n-1) d}{2}\right) ,
\end{equation}

in particular if $nd = 2\pi m$ for some integer $m$ and $0 < \frac{d}{2} < \pi$ the sum is zero independent of $\alpha$.

\vspace{0.2in}
\begin{lemma} 
If $1 \leq r \leq k $ for $k$ fixed and $a$ any number 
$$
\sum_{j=0}^{2 k+1} \cos ^{2 r}\left(a+j \cdot \frac{2 \pi}{2 k+2}\right) = (2 k+2) \cdot \frac{1}{2^{2 r}} \cdot\left(\begin{array}{c}
2 r \\
r
\end{array}\right)
$$

\end{lemma}

\begin{proof}

\begin{align*}
& \sum_{j=0}^{2 k+1} \cos ^{2 r}\left(a+j \cdot \frac{2 \pi}{2 k+2}\right)=\sum_{j=0}^{2 k+1} \cos ^{2 r}\left( a+ j \cdot \frac{\pi}{k+1}\right)&\\\\&=\sum_{j=0}^{2 k+1} \left\{\frac{1}{2^{2 r}}\left(\begin{array}{c}
2 r \\
r
\end{array}\right)+\frac{2}{2^{2 r}} \sum_{l=0}^{r-1}\left(\begin{array}{c}
2 r \\
l
\end{array}\right) \cos \left((2 r-2 l)\left(a+j \cdot \frac{\pi}{k+1}\right)\right)\right.&\\\\&=\sum_{j=0}^{2 k+1} \frac{1}{2^{2 r}}  \left(\begin{array}{c}
2 r \\
r
\end{array}\right)+\sum_{j=0}^{2 k+1} \frac{2}{2^{2 r}} \sum_{l=0}^{r-1}\left(\begin{array}{c}
2 r \\
l
\end{array}\right) \cos \left(a \cdot (2r-2l) +j \cdot \frac{(r -l ) \cdot 2 \pi}{k+1}\right)&\\\\&=  \sum_{j=0}^{2 k+1} \frac{1}{2^{2 r}}  \left(\begin{array}{c}
2 r \\
r
\end{array}\right)+\sum_{l=0}^{r-1} \frac{2}{2^{2 r}}\left(\begin{array}{c}
2 r \\
l
\end{array}\right) \sum_{j=0}^{2k+1} \cos \left(a \cdot (2r-2l) +j \cdot \frac{(r -l ) \cdot 2 \pi}{k+1}\right)&\\&=  \sum_{j=0}^{2 k+1} \frac{1}{2^{2 r}}  \left(\begin{array}{c}
2 r \\
r
\end{array}\right) =(2 k+2) \cdot \frac{1}{2^{2 r}} \cdot\left(\begin{array}{c}
2 r \\
r
\end{array}\right). &
\end{align*}

The inner sum in the next to the last line is zero using equation (5.1) with $n=2k+2, d=\frac{(r-l) 2 \pi}{k+1}$ and noting $1 \leq r-l \leq r \leq k$.

\end{proof}

\end{document}